\documentclass{amsart}

\usepackage[T1]{fontenc}
\usepackage[utf8]{inputenc}
\usepackage{amsmath,amsfonts,amsthm,amssymb,amsxtra,bbm}
\usepackage{fourier,setspace,graphicx,color,pdflscape}
\usepackage[colorlinks=true]{hyperref}
\hypersetup{urlcolor=blue, linkcolor=blue, citecolor=red, anchorcolor=blue]}

\newtheorem{theorem}{Theorem}
\newtheorem{proposition}[theorem]{Proposition}
\newtheorem{lemma}[theorem]{Lemma}
\newtheorem{corollary}[theorem]{Corollary}
\newtheorem{remark}[theorem]{Remark}

\newcommand{\R}{{\mathbb R}}
\newcommand{\C}{{\mathbb C}}
\newcommand{\N}{{\mathbb N}}

\newcommand{\be}[1]{\begin{equation}\label{#1}}
\newcommand{\ee}{\end{equation}}
\renewcommand{\(}{\left(}
\renewcommand{\)}{\right)}
\newcommand{\irdx}[1]{\int_{\R^d}{#1}\,dx}
\newcommand{\irdv}[1]{\int_{\R^d}{#1}\,dv}
\newcommand{\irdxv}[1]{\iint_{\R^d\times\R^d}{#1}\,dx\,dv}
\newcommand{\nrm}[2]{\|{#1}\|_{#2}}
\newcommand{\Nrm}[2]{\left\|{#1}\right\|_{#2}}
\newcommand{\nrmx}[2]{\nrm{#1}{\mathrm L^{#2}(\R^d)}}
\newcommand{\nrmv}[2]{\nrm{#1}{\mathrm L^{#2}(\R^d)}}
\newcommand{\nrmxv}[2]{\nrm{#1}{\mathrm L^{#2}(\R^d\times\R^d)}}
\newcommand{\Nrmxv}[2]{\Nrm{#1}{\mathrm L^{#2}(\R^d\times\R^d)}}
\newcommand{\gammastar}{\gamma_{\!\star}}

\makeatletter\@namedef{subjclassname@2020}{\textup{2020} Mathematics Subject Classification}
\makeatother

\definecolor{darkgreen}{rgb}{.0,.4,.2}

\usepackage{etoolbox}
\newcounter{taggedeq}
\pretocmd{\equation}{\stepcounter{taggedeq}}{}{}

\RequirePackage{silence}
\title[The fundamental solution of a nonlinear kinetic Fokker-Planck equation]{The fundamental solution of\\ a nonlinear kinetic Fokker-Planck equation}

\author{Giovanni Brigati}
\author{Guillaume Carlier}
\author{Jean Dolbeault}

\address{\hspace*{-12pt} G.~Brigati: \rm IRIMAS, Université de Haute-Alsace, 18 Rue des Frères Lumière, 68093 Mulhouse, France}
\email{giovanni.brigati1995@gmail.com}

\address{\hspace*{-12pt} G.~Carlier: \rm CEREMADE (CNRS UMR n$^\circ$ 7534), PSL university, Universit\'e Paris-Dauphine, Place de Lattre de Tassigny, 75775 Paris 16, France, and Inria-Paris, MOKAPLAN, Paris, France}
\email{carlier@ceremade.dauphine.fr}

\address{\hspace*{-12pt} J.~Dolbeault: \rm CEREMADE (CNRS UMR n$^\circ$ 7534), PSL university, Universit\'e Paris-Dauphine, Place de Lattre de Tassigny, 75775 Paris 16, France}
\email{dolbeaul@ceremade.dauphine.fr}

\begin{document}

\begin{abstract}
This paper is devoted to a fundamental solution of a nonlinear kinetic equation involving a porous medium or fast diffusion operator acting on velocities. Such a nonlinearity has interesting scaling properties, which result in a self-similar behaviour of the fundamental solution. Here fundamental solution means a solution with Dirac distribution initial datum which moreover governs the large time asymptotics of a large class of solutions. Using a self-similar change of variables, the equation becomes a nonlinear kinetic Fokker-Planck equation with harmonic confinement and the intermediate asymptotics regime is transformed into a stability property of a special stationary solution, which attracts the solutions for large times. In the homogeneous case (pure nonlinear diffusion), the problem is reduced to a classical nonlinear diffusion equation with Barenblatt-Pattle self-similar profiles. This beautiful structure is preserved at kinetic level, with remarkable consequences for relative entropy estimates, detailed intermediate asymptotics and nonlinear diffusion limits.
\end{abstract}

\date{\today}



\maketitle
\thispagestyle{empty}

\section{Introduction and main results}\label{Sec:Intro}

The computation of the Green function for the linear Vlasov-Fokker-Planck equation by A.~Kolmogorov in~\cite{kolmogoroff_zufallige_1934} was the starting point of the hypo-ellipticity theory developed later by L.~H\"ormander in~\cite{MR222474}. Barenblatt-Pattle self-similar solutions of~\cite{MR46217,MR114505} play the role of \emph{fundamental solutions} for the porous medium and the fast diffusion equations (see for instance~\cite{vazquez2023surveymassconservationrelated}), in the sense that they have Dirac distributions as initial data and govern the large time behaviour of finite mass nonnegative solutions in the appropriate range of parameters. Over the years, these solutions became an essential tool in the theory of nonlinear diffusion equations, see~\cite{MR2282669,MR2286292}. Our aim is to extend this notion of fundamental solution to a class of nonlinear and non-homogeneous kinetic equations, and draw some consequences. As often in nonlinear analysis, our purpose is not to cover a large class of equations, but to focus on a well-chosen equation which captures important properties. We gain insight by a detailed study of the nonlinear phenomena which explain the qualitative and quantitative behaviours of the solutions, opening new perspectives on large time behaviours and intermediate asymptotics, diffusion limits and quantitative convergence estimates (hypocoercivity) in nonlinear kinetic theory.

\medskip Let us consider the nonlinear kinetic equation
\be{Eqn:f}
\partial_tf+v\cdot\nabla_xf=\Delta_vf^m\,,\quad(t,x,v)\in\R^+\times\R^d\times\R^d\,.
\ee
We generalise to the case $m\neq1$ a few results known for $m=1$. Inspired by the Green function obtained by A.~Kolmogorov in~\cite{kolmogoroff_zufallige_1934} in the linear case $m=1$, and the Barenblatt-Pattle functions in the homogeneous case (\emph{i.e.}, when $f$ does not depend on $x$ and solves a nonlinear diffusion equation), we look for self-similar solutions of~\eqref{Eqn:f} with finite mass and Dirac distribution $\boldsymbol\delta$ as initial data. Let us introduce the \emph{pressure} variable
\[
\mathsf P:=\frac m{1-m}\,f^{m-1}
\]
and, on the support of $f$, rewrite~\eqref{Eqn:f} as
\be{eq:P}
\partial_t\mathsf P=(1-m)\,\mathsf P\,\Delta_v\mathsf P-\,|\nabla_v\mathsf P|^2-v\cdot\nabla_x\mathsf P\,.
\ee
The differential operator in the right-hand side preserves second order polynomials in~$(x,v)$ with $t$-dependent coefficients. An elementary computation shows that
\be{eq:Pstar}
\mathsf P_{\kern-1pt\star}(t,x,v)=\beta(t)+\frac{(1+A)}{2\,(1-A)\,t}\(\left|v-\frac x{(1-A)\,t}\right|^2+A\,\left|\frac x{(1-A)\,t}\right|^2\)
\ee
with
\be{A}
A=\frac{1+d-d\,m}{3-d+d\,m}
\ee
is a solution of~\eqref{eq:P}. As we shall see in Lemma~\ref{Lem:Kolmogorov}, if
\be{Range}
m\in(m_1,1)\cup(1,m_2)\quad\mbox{with}\quad m_1:=1-\frac1d\,,\quad m_2:=1+\frac1d\,,
\ee
the corresponding \emph{fundamental solution} of~\eqref{Eqn:f} is given by
\be{fstar}
f_\star(t,x,v)=\(\tfrac{1-m}m\,\mathsf P_{\kern-1pt\star}(t,x,v)\)_+^\frac1{m-1}\,.
\ee
Here $\beta(t)=\big((1-A)\,t\big)^{2\,(1-m)/(m-m_1)}\,\gammastar$ is an explicit normalization constant (see~\eqref{Eq:Mass} in Section~\ref{Sec:Mass} after using the change of variables~\eqref{SSCoV}-\eqref{R} with $R_0=0$) which can also be computed directly using $\nrmxv{f_\star(t,\cdot,\cdot)}1=1$. We have not found such a solution in the current literature, but one recovers Kolmogorov's solution~\cite{kolmogoroff_zufallige_1934} as $m\to1$ (see \hbox{Section~\ref{Sec:Kolmogorov}}).

\medskip Our main result is that the \emph{fundamental solution} $f_\star$ is an interesting special solution which governs the large time behaviour of all solutions of~\eqref{Eqn:f} with initial datum $f_0$ such that
\be{Initial:f}
\begin{aligned}
&f_0\in\mathrm L^1(\R^d\times\R^d)\,,\quad\nrmxv{f_0}1=1\,,\\
&g_{\gamma_1}(x,v-x)\le f_0(x,v)\le g_{\gamma_2}(x,v-x)
\end{aligned}
\ee
for any $(x,v)\in\R^d\times\R^d$, with $(1-m)\,\gamma_1>0$, $(1-m)\,\gamma_2>0$, $\gamma_1>\gamma_2$ and
\be{g}
g_\gamma(x,v):=\(\tfrac{1-m}m\,\(\gamma+\tfrac{1+A}2\,|v|^2+\tfrac{1+A}2\,A\,|x|^2\)\)_+^\frac1{m-1}\,.
\ee
In simpler words, we assume that~$f_0$ is nonnegative, of unit mass, and trapped between two self-similar profiles of masses $M_i=\irdxv{g_{\gamma_i}(x,v)}$ with \hbox{$M_1<1<M_2$}. Assumption~\eqref{Initial:f} is a technical condition which simplifies a lot the proofs but can be removed to the price of more complicated estimates as in~\cite{bouin2026nonlinearkineticfokkerplanckequations,brigati2026nonlinearkineticfokkerplanckequations}.
\begin{theorem}\label{Thm:Main} Let $d\ge1$ and assume~\eqref{Range}. Under Condition~\eqref{Initial:f}, there is a unique solution $f$ of~\eqref{Eqn:f} with initial datum~$f_0$ and
\[
\lim_{t\to+\infty}\Nrmxv{f(t,\cdot,\cdot)-f_\star(t,\cdot,\cdot)}1=0\,.
\]
\end{theorem}
\noindent The notion of solution will be specified in Section~\ref{Sec:Existence}. Since~\eqref{Eqn:f} conserves mass for any $t\ge0$ but $f_\star(t,\cdot,\cdot)\to0$ as $t\to+\infty$ in $\mathrm L^1_{\rm loc}(\R^d\times\R^d)$, the result of Theorem~\ref{Thm:Main} is an \emph{intermediate asymptotics} estimate. Even the limit case $m=1$ of Theorem~\ref{Thm:Main} seems new, although it could be proved directly using the Green function found by Kolmogorov and classical computations: see~\cite{MR222474} or~\cite[Appendix~A]{bouin2020hypocoercivity}. The level lines of the \emph{fundamental solution}~$f_\star$ given by~\eqref{fstar} are ellipses, which rotate and expand in the phase space $\R^d\times\R^d\ni(x,v)$ as $t$ increases. See Fig.~\ref{F1} for plots of the ellipse $\mathfrak E_m(t)$ which encloses half of the mass of~$f_\star$.
\setlength\unitlength{1cm}
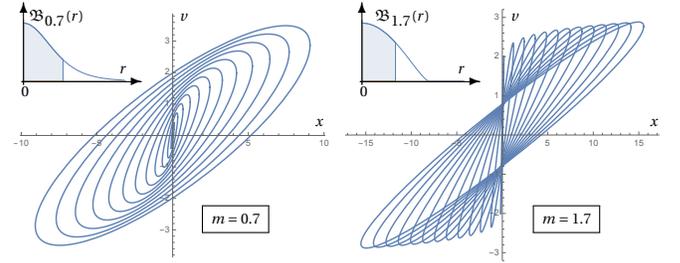
\begin{figure}[ht]
\begin{center}
\begin{picture}(8.5,3.5)
\put(0,0.04){\includegraphics[width=4.2cm]{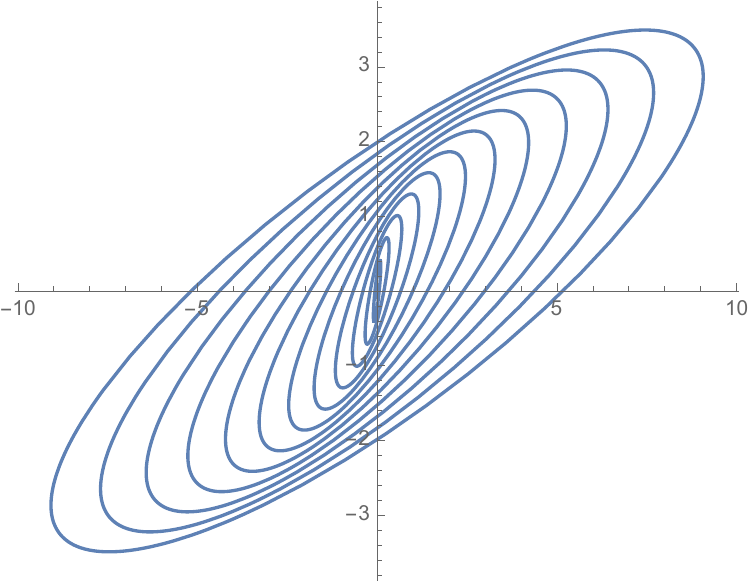}}
\put(0.1,2.35){\includegraphics[width=1.4cm]{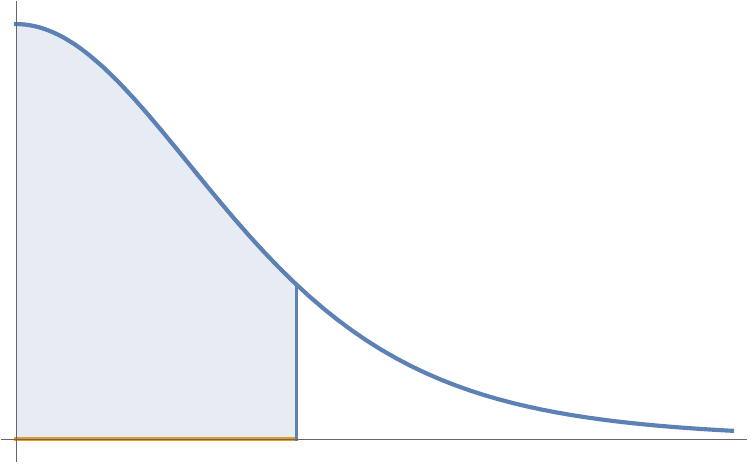}}
\put(4.4,0){\includegraphics[width=4.2cm]{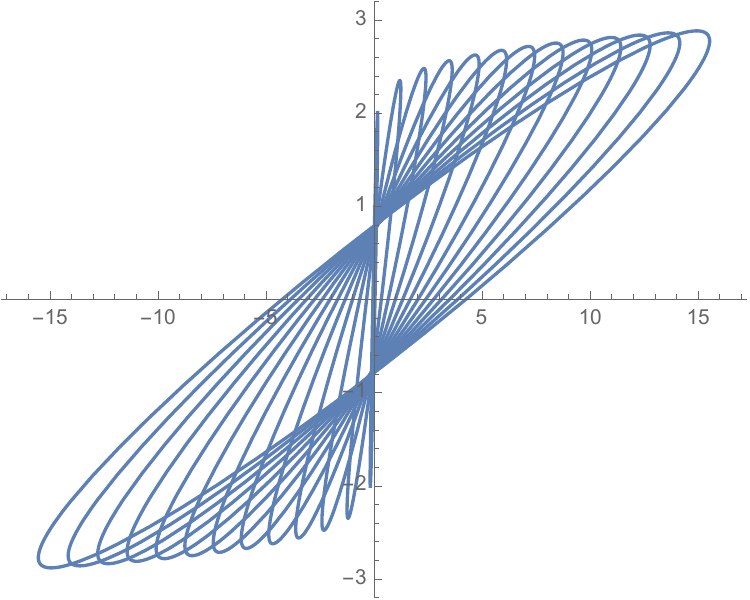}}
\put(4.6,2.35){\includegraphics[width=1.4cm]{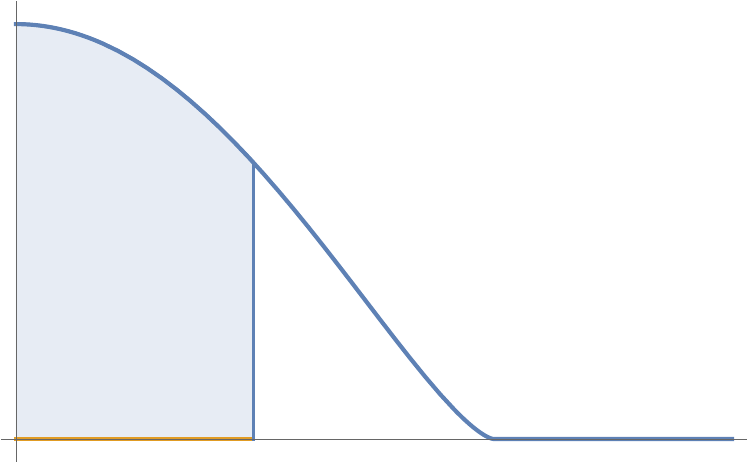}}
\put(4,1.8){\tiny{$x$}}
\put(8.45,1.8){\tiny{$x$}}
\put(2.2,3.2){\tiny{$v$}}
\put(6.6,3.2){\tiny{$v$}}
\put(0.2,3.2){\tiny{$\mathfrak B_{0.7}(r)$}}
\put(0.1,2.393){\vector(1,0){1.6}}
\put(0.128,2.35){\vector(0,1){1.1}}
\put(0.1,2.2){\tiny{$0$}}
\put(1.4,2.5){\tiny{$r$}}
\put(2.5,0.5){\fbox{\tiny{$m=0.7$}}}
\put(4.8,3.2){\tiny{$\mathfrak B_{1.7}(r)$}}
\put(4.6,2.393){\vector(1,0){1.6}}
\put(4.628,2.35){\vector(0,1){1.1}}
\put(4.6,2.2){\tiny{$0$}}
\put(5.9,2.5){\tiny{$r$}}
\put(6.9,0.5){\fbox{\tiny{$m=1.7$}}}
\end{picture}
\end{center}
\caption{\small\label{F1}In dimension $d=1$, the ellipse $\mathfrak E_m(t)$ is represented in the phase space $\R^d\times\R^d\ni(x;v)$ at times $t=0.1$, $t=0.6$, $1.1$\ldots $6.1$ for $m=0.7$ (left) and $m=1.7$ (right). In the upper left corners, the Barenblatt profiles $\mathfrak B_m(r):=\big(1\pm r^2\big){}_+^{1/(m-1)}$ are shown, with shaded areas corresponding to half of the mass. Here $\pm$ denotes the sign of $(1-m)$ and~$r$ is such that $r^2=A\,|x|^2+|v|^2$: the function $f_\star(t,\cdot,\cdot)$ has compact support if $m>1$.}
\end{figure}

\medskip As in the study of porous media and fast diffusions equations, it is convenient to replace \emph{intermediate asymptotics} estimates by estimates of convergence to a stationary solution after a \emph{time-dependent rescaling}. The expression of $f_\star$ suggests to introduce \emph{self-similar variables} and consider $g$ such that
\be{SSCoV}
f(t,x,v)=R(t)^{-d\,(1+A)}\,g\(\log R(t),\frac x{R(t)},\frac v{R(t)^A}-\frac x{R(t)}\)
\ee
with
\be{R}
R(t)=\(R_0^{1-A}+(1-A)\,t\)^\frac1{1-A}\quad\mbox{and}\quad A=\frac{1+d-d\,m}{3-d+d\,m}
\ee
where $R_0=R(0)\ge0$, and $A=A(m)$ is given by~\eqref{A}. It is indeed a known fact, at least in plasma physics, that by rewriting a solution $f$ of a Vlasov equation~as
\[
f(t,x,v)=\big(R(t)\,R'(t)\big)^{-d}\,g\(\tau(t),\frac x{R(t)},\frac v{R'(t)}-\frac x{R(t)}\)\,,
\]
the Galilean invariance corresponding to the free transport operator $\partial_t+v\cdot\nabla_x$ is preserved, up to a friction term and a confinement term corresponding to a harmonic potential. The key idea is indeed to introduce scales in $x$ and $v$ given respectively by $R$ and $R'$ for some function $t\mapsto R(t)$ to be determined. Here $'$ denotes the $t$-derivative. See~\cite[Section~3]{MR1830948} and~\cite{Burgan-Feix-Fijalkow-Munier-1983} for more details. We can implement such a change of variables in~\eqref{Eqn:f}. With the ansatz that $\tau'$, $R'/R$, $R''/R'$ and the term coming from the nonlinear diffusion, that is, $R^{d\,(1-m)}\,(R')^{d\,(1-m)-2}$, are all proportional, so that the rescaled equation has no coefficients depending explicitly on $t$, it is indeed an elementary computation to check that~\eqref{R} provides us with a solution of the ordinary differential equation
\[
\frac1R\,\frac{dR}{dt}=R^{d\,(1+A)(1-m)-2\,A}\,.
\]
With the change of variables~\eqref{SSCoV}-\eqref{R}, if $f$ is a solution of~\eqref{Eqn:f}, then $g$ solves
\be{Eqn:g}
\partial_tg+v\cdot\nabla_xg-A\,x\cdot\nabla_vg=\Delta_vg^m+(1+A)\,\nabla_v\cdot(v\,g)\,.
\ee
If $R_0=1$, the initial datum is
\be{g0}
g(0,x,v)=g_0(x,v):=f_0(x,v+x)
\ee
if $f$ solves~\eqref{Eqn:f} with initial datum $f_0$. It is also elementary to check that~\eqref{Eqn:g} admits a nonnegative stationary solution $g_\gamma$ for some $\gamma\in\R$ such that $(1-m)\,\gamma>0$. Notice that the finite mass condition on~$g_\gamma$ induces a restriction on $m$, which goes as follows.
\begin{lemma}\label{Lem:Kolmogorov} Let $d\ge1$. Then $g_\gamma$ is a function in $\mathrm L^1(\R^d\times\R^d)$ if and only if $m$ satisfies~\eqref{Range}. In that range,~\eqref{Eqn:f} admits a \emph{fundamental solution} $f_\star$ given by~\eqref{fstar}, \emph{i.e.}, by~\eqref{g} with $\gamma=\gamma_\star$ using~\eqref{SSCoV} and~\eqref{R} with $R_0=0$.\end{lemma}
\noindent The Herrero-Pierre exponent corresponding to the threshold case of integrability of Barenblatt profiles in $\R^d$ is $m_c:=(d-2)/d$ according to~\cite{MR797051}. The exponent $m_1$ in Lemma~\ref{Lem:Kolmogorov}, that is $=1-1/d=(2\,d-2)/(2\,d)$, is the classical Herrero-Pierre exponent in $\R^{2d}\approx\R^d\times\R^d$, while the limitation $m<m_2$ arises from \hbox{$A(m)>0$}. For more detailed references on nonlinear diffusions without transport, we refer to~\cite{BDNS2025}, and to~\cite {Blanchet:2009sf} for an extended range of $m$. From here on, we shall assume that~\eqref{Range} holds, take $\gamma=\gammastar$ in~\eqref{g} and write $g_\star:=g_{\gamma_\star}$, \emph{i.e.},
\be{gstar}
g_\star(x,v):=\(\tfrac{1-m}m\,\(\gamma_\star+\tfrac{1+A}2\,|v|^2+\tfrac{1+A}2\,A\,|x|^2\)\)_+^\frac1{m-1}\,.
\ee
Our goal is to prove Theorem~\ref{Thm:Main} and establish large time asymptotic properties of the solutions of~\eqref{Eqn:f} using nonlinear parabolic methods in Section~\ref{Sec:Mass-Recaling} and relative entropy methods applied to the solutions of~\eqref{Eqn:g} in Section~\ref{Sec:Entropy}. At formal level, we shall also sketch consequences on the diffusion limit and the linearization of~\eqref{Eqn:g} around~$g_\star$ in Section~\ref{Sec:Further}.

\medskip From a more physical perspective,~\eqref{Eqn:f} models the evolution of a particle distribution in phase space $(x,v)\in\R^d\times\R^d$, where $x$ represents position, $v$ velocity and $t\ge0$ is the time. The term $v\cdot\nabla_xf$ describes free transport, while $\Delta_vf^m$ models nonlinear diffusion in velocity with two main regimes: fast diffusion if $m<1$ and porous medium regime if $m>1$. The lower bound $m_1=1-1/d$ and the upper bound $m_2=1+1/d$ ensure that the Barenblatt-Pattle profiles are integrable in $\R^d$. The fundamental solution $f_\star$ constructed in this paper describes how an initial concentration of particles at a single point in phase space evolves over time, spreading out in a self-similar manner. The large-time asymptotics result (Theorem~\ref{Thm:Main}) shows that, regardless of the initial configuration but under suitable conditions, the solution eventually behaves like $f_\star$ for large values of $t$, with the particles spreading according to the self-similar profiles illustrated in Fig.~\ref{F1}.

\medskip Concerning entropy methods applied to fast diffusion and porous medium equations, we shall refer to~\cite{Newman1984,Ralston1984,MR1940370,zbMATH01546987,MR1940370,Blanchet:2009sf} for early contributions,~\cite{zbMATH06587726} for a detailed analysis from the PDE point of view, and~\cite{BDNS2025} for a recent overview of the topic, with detailed references therein. In the context of kinetic equations, we can point~\cite{MR1677677} for the use of free energy or relative entropy notions in kinetic theory as well as confinement and boundedness from below of thermodynamical Lyapunov functionals in nonlinear models,~\cite{MR2065020} for an early introduction to entropies in kinetic and parabolic equations, and~\cite{MR2338354} in the more specific case of Barenblatt-Pattle local equilibria. In~\cite{MR2065020}, convex duality between profiles of local equilibria and generalized entropies is detailed as well as some sources of inspiration in physics and gravitational theory in astrophysics. Understanding nonlinear diffusions as gradient flows of entropy functional for appropriate notions of distances goes back to~\cite{MR1842429}. For nonlinear kinetic Fokker-Planck equations, this has been tackled recently in~\cite{brigati2026nonlinearkineticfokkerplanckequations}. For further historical references and the use of splitting schemes combining transport and Wasserstein gradient steps in order to build solutions, see~\cite{MR2048566}.

After a preliminary version of this paper was posted on arXiv (\href{https://arxiv.org/abs/2603.26650}{2603.26650}), several other papers appeared, which are concerned with various properties of~\eqref{Eqn:f}: gradient flow structure in~\cite{brigati2026nonlinearkineticfokkerplanckequations}, existence result for large classes of initial data~\cite{brigati2026nonlinearkineticfokkerplanckequations,bouin2026nonlinearkineticfokkerplanckequations,hao2026kineticfokkerplanckequationsnonlinear}, diffusion limit~\cite{bouin2026nonlinearkineticfokkerplanckequations} and large time asymptotics~\cite{carrillo2026conditionalhypocoercivitynonlinearkinetic}. In all these papers, the self-similar properties of~\eqref{Eqn:f} and the fundamental solution $f_\star$, directly or indirectly, play a significant role.

\section{Mass, existence, contraction and comparison}\label{Sec:Mass-Recaling}

\subsection{The mass parameter}\label{Sec:Mass}

Let $f$ be a nonnegative function in $\mathcal C\big(\R^+;\mathrm L^1(\R^d\times\R^d)\big)$, $M>0$ a real number and
\[
f_M(t,x,v):=M\,f\(M^{\,2\,\zeta}\,t,M^{\,\zeta}\,x,M^{\,-\,\zeta}\,v\)\quad\mbox{with}\quad\zeta:=-\,\frac14\,(1-m)\,.
\]
\begin{lemma}\label{Lem:mass} With the above notation, $f_M$ solves~\eqref{Eqn:f} if and only if~$f$ solves~\eqref{Eqn:f}.\end{lemma}
\noindent Considering solutions of~\eqref{Eqn:f} or~\eqref{Eqn:g} with initial data $f_0$ such that $\nrmxv{f_0}1=1$ is therefore not a restriction. Let $\pm$ denote the sign of $(1-m)$. The mass normalization $\nrmxv{g_\star}1=1$ determines $\gamma=\gammastar$ in~\eqref{Eqn:g} such that
\be{Eq:Mass}
\tfrac{2^d}{\big(\sqrt A\,(1+A)\big)^d}\,\big|\tfrac{1-m}m\big|^\frac1{m-1}\,\gammastar^{\,d\,\frac{m-m_1}{m-1}}\kern-2pt\int_{\R^{2d}}\big(1\pm\,|z|^2\big)_+^\frac1{m-1}\,dz=1
\ee
after the simple change of variables $z=\sqrt{\frac{1+A}{2\,\gamma}}\,\big(\sqrt A\,x,v\big)$. Notice that the integral in~\eqref{Eq:Mass} can be expressed in terms of the Euler \emph{Beta} function.

\subsection{Existence of a solution}\label{Sec:Existence}

\begin{proposition}\label{Prop:Existence} Under the assumptions of Theorem~\ref{Thm:Main}, there is a weak solution $f$ of~\eqref{Eqn:f} with initial datum $f_0$ in the space $\mathcal C\big(\R^+;\mathrm L^1(\R^d\times\R^d)\big)$ which conserves mass.\end{proposition}
\begin{proof} After specifying our notion of weak solutions, we shall introduce an approximation scheme based on splitting, establish some compactness of these approximations and prove that any limit point of the scheme is a weak solution.
\\[4pt]
$\bullet$ \emph{Weak solutions and functional setting}. For any nonnegative $f_0$ satisfying~\eqref{Initial:f}, we have $f_0\in\mathrm L^\infty\cap\mathrm L^1\big(\R^d\times\R^d,(1+|x|^2+|v|^2)^\beta\,dx\,dv\big)$ with
\[
\beta\in\big(0,\tfrac1{1-m}-d\big)\;\mbox{if}\;m\in(m_1,\widetilde m_1)\;\mbox{with}\;\widetilde m_1:=\tfrac d{d+1}\,,\;\mbox{and}\;\beta=1\;\mbox{if}\;m\in[\widetilde m_1,1)\cup(1,m_2)\,.
\]
The integrability condition $f_0\in\mathrm L^1\big(\R^d\times\R^d,(1+|x|^2+|v|^2)\,dx\,dv\big)$ corresponds to $m>\widetilde m_1$. We have $f_0\in\mathrm L^m(\R^d\times\R^d)$ by assumption if $m\in(1,m_2)$. If $m\in(m_1,1)$, $\beta$ is as above and $f\in\mathrm L^1\big(\R^d\times\R^d,(1+|x|^2+|v|^2)^\beta\,dx\,dv\big)$ is nonnegative, H\"older's inequality yieds
\be{eq:holderbeta}
\begin{aligned}
&\hspace*{-8pt}\irdxv{\(1+|x|^2+|v|^2\)^{\beta-1}\,f^m}\\
&\le\(\irdxv{\(1+|x|^2+|v|^2\)^\beta\,f}\)^m\(\irdxv{\(1+|x|^2+|v|^2\)^{\beta-\frac1{1-m}}}\)^{1-m}\hspace*{-8pt}
\end{aligned}
\ee
so that $f^m\in\mathrm L^1\big((1+|x|^2+|v|^2)^{\beta-1}\big)$ since $\big(1+|x|^2+|v|^2\big)^{\beta-1/(1-m)}$ is integrable. Hence $f_0^m\in \mathrm L^1\big((1+|x|^2+|v|^2)^{\beta-1}\big)$ for any exponent $m\in(m_1,1)\cup (1, m_2)$.

Let us assume that $f_0\in\mathrm L^\infty(\R^d\times\R^d)\cap\mathrm L^1\big(\R^d\times\R^d,(1+|x|^2+|v|^2)^\beta\,dx\,dv\big)$ is nonnegative. We construct a weak solution $f$ of~\eqref{Eqn:f} with initial datum $f_0$ on the time interval $[0,1]$, that is a nonnegative function
\[
f\in\mathcal C\big([0,1];\mathrm L^1(\R^d\times\R^d)\big)\cap\mathrm L^\infty\big((0,1);\mathrm L^\infty(\R^d\times\R^d)\cap\mathrm L^1\big(\R^d\times\R^d,(1+|x|^2+|v|^2)^\beta\,dx\,dv\big)\big)
\]
such that, for any $t\in [0,1]$ and $\phi\in \mathcal{C}_c^2(\R^d\times\R^d)$,
\be{eq:defweaksol}
\irdxv{\kern-15pt\phi\,f(t)}=\irdxv{\kern-15pt\phi\,f_0}+\int_0^t\irdxv{\(v\cdot\nabla_x\phi\,f(s)+\Delta_v\phi\,f^m(s)\)}\,ds\kern-6pt
\ee
with the short-hand notation $f(t)=f(t,\cdot,\cdot)$. Thanks to~\eqref{eq:holderbeta}, if $f$ is such a solution, then $f^m\in\mathrm L^\infty((0,1);\mathrm L^1\big((1+|x|^2+|v|^2)^{\beta-1}\big)$ and $f(1)\in\mathrm L^\infty(\R^d\times\R^d)\cap\mathrm L^1\big(\R^d\times\R^d,(1+|x|^2+|v|^2)^\beta\,dx\,dv\big)$. Hence, if we can build a solution on the time interval $[0,1]$, this construction can be iterated so as to ultimately produce a solution on the whole of $\R^+$.
\\[4pt]
$\bullet$ \emph{Splitting scheme}. For $n\in \N\setminus\{0\}$, we can solve on the time interval $[0,1/n]$
\begin{align*}\tag{D}
&\frac{\partial f}{\partial t}=2\,\Delta_v f^m\quad&\mbox{if}\quad0\le t\le\tfrac1{2\,n}\,,\\
\tag{T}
&\frac{\partial f}{\partial t}+2\,v\cdot\nabla_x f=0\quad&\mbox{if}\quad\tfrac1{2\,n}\le t\le\tfrac1n\,,
\end{align*}
with an initial datum $f_0$ as above. Indeed, the diffusion step (D) is a standard nonlinear diffusion equation on each fiber $(x, \cdot)$ and the second step of free transport (T) is explicitly solved by $f(t,x,v)=f\big(1/(2\,n), x-(t-1/(2\,n))v,v\big)$. Denoting by $(S_n(t)\,f_0)_{t\in [0,1/n]}$ this solution, we define $f_n$ on the time interval $[0,1]$ by $f_n(t+k/n):=S_n(t)\,f(k/n)$, $t\in [0,1/n]$, $k\in \{0, \ldots, n-1\}$. Next we have $f_n \in \mathcal C\big([0,1];\mathrm L^1(\R^d\times\R^d)\big)$. Moreover, since both diffusion and transport steps preserve nonnegativity, mass, and do not increase the $\mathrm L^\infty(\R^d\times\R^d)$ norm, we have $f_n \in\mathrm L^\infty\big((0,1);\mathrm L^1\cap\mathrm L^\infty(\R^d\times\R^d)\big)$.

If $m\in(1, m_2)$, $t\mapsto \irdxv{f_n^m(t, x,v)}$ is nonincreasing and then
\[
\sup_{\begin{array}{c}\scriptstyle n\in\N\\[-4pt]\scriptstyle t\in [0,1]\end{array}}\irdxv{\((1+|x|^2+|v|^2\)f_n+f_n^m)}\le C\,.
\]
For $m\in(m_1, 1)$, we proceed as follows. Defining $M_\beta(f_n):=\big\|\(1+|x|^2+|v|^2\)^\beta\,f_n\big\| _{\mathrm L^1}$ and $M_{\beta-1}(f_n^m):=\big\|\(1+|x|^2+|v|^2\)^{\beta-1}\,f_n^m\big\|_{\mathrm L^1}$, on an interval of free transport $J_i^n$, that is,
\[
J_i^n:=\left[\tfrac{2\,i+1}{2\,n},\tfrac{2\,i+2}{2\,n}\right]\quad\mbox{with}\quad i\in\{0,\cdots,n-1\}\,,
\]
we have
\[
\sup_{t\in J_i^n} M_\beta\big(f_n(t)\big)\le\big(1+\tfrac3n\big)^\beta\,M_\beta\(f_n\big(\tfrac{2\,i+1}{2\,n}\big)\)
\]
and on an interval $I_i^n$ of diffusion, that is,
\[
I_i^n=\left[\tfrac{2\,i}{2\,n},\tfrac{2\,i+1}{2\,n}\right]\quad\mbox{with}\quad i\in\{0,\cdots,n-1\}\,,
\]
we have
\begin{align*}
\frac{d}{dt} M_\beta\big(f_n(t)\big)&= 2\irdxv{\Delta_v\(1+|x|^2+|v|^2\)^\beta\,f_n^m(x,v)}\\
&\le C \irdxv{\(1+|x|^2+|v|^2\)^{\beta-1}\,f_n^m(x,v)}\\
&\le C\,M_{\beta}\big(f_n(t)\big)^m
\end{align*}
where we used~\eqref{eq:holderbeta} in the third line. From these considerations, we deduce
\be{eq:basic0f}
\sup_{\begin{array}{c}\scriptstyle n\in\N\\[-4pt]\scriptstyle t\in [0,1]\end{array}}\Big(M_{\beta}\big(f_n(t)\big)+M_{\beta-1}\big(f_n^m(t)\big)+\| f_n(t)\|_{\mathrm L^\infty}\Big)\le C\,.
\ee
Denoting $I^n:=\cup_{i=0}^{n-1} I_i^n$ and $J^n:=\cup_{i=0}^{n-1} J_i^n$, we observe that by construction $f_n$ solves
\[
\frac{\partial f_n}{\partial t}+2\,\mathbb 1_{J^n}(t)\,v\cdot\nabla_xf_n=2\,\mathbb 1_{I^n}(t)\,\Delta_vf_n^m
\]
with initial condition $f_0$. Let us introduce the space of test-functions
\[
X:=\left\{\phi\in\mathcal C^2(\R^d\times\R^d)\,:\,\|\phi\|_{X}:=\big\|\big(1+|x|^2+|v|^2\big)\,(\phi,\nabla_{x,v}\phi,D^2_{x,v}\phi)\big\|_{\mathrm L^\infty}<+\infty\right\}\,.
\]
We deduce from the bound~\eqref{eq:basic0f}, a standard cutoff argument and Lebesgue's dominated convergence Theorem that for every $\phi\in X$ and $t\in [0,1]$, there holds
\be{eq:weakformforfn}
\begin{aligned}
&\irdxv{\kern-8pt\phi\,f_n(t)}\\
&=\irdxv{\kern-8pt\phi\,f_0}+\,2\int_0^t\irdxv{\kern-8pt\(\mathbb1_{J^n}(s)\,v\cdot\nabla_x\phi\,f_n(s)+\mathbb1_{I^n}(s)\,\Delta_v\phi\,f_n^m(s)\)}\,ds\,.\kern-6pt
\end{aligned}
\ee
$\bullet$ \emph{Compactness in $\mathrm L^1(\R^d\times\R^d)$}. For $g\in \mathrm L^1(\R^d\times\R^d)$ and $\delta\geq 0$, we define
\[
\omega_g(\delta):=\sup_{y, w\,:\,|y|\le\delta,\,|w|\le\delta} \| g(y+\cdot, v+\cdot)-g\|_{\mathrm L^1(\R^d\times\R^d)}\,.
\]
Among good properties of $f_n$ inherited from diffusion and free transport steps, we obtain a strong control of the modulus $\omega_{f_n(t)}$. Indeed, during the diffusion step, the invariance of (D) by translation in both position and velocity together with the $\mathrm L^1$-contraction for nonlinear diffusions~\cite{MR2286292} yields
\[
\sup_{t\in I_i^n} \omega_{f_n(t)}\le\omega_{f_n(i/n)}\,.
\]
Whereas, on an interval of free transport $J_i^n$, since
\[
f_n(t,x,v):=f_n\Big(\tfrac{2\,i+1}{2\,n},x-2\(t-\tfrac{2\,i+1}{2\,n}\)\,v,v\Big)\quad\forall\,t\in J_i^n\,,
\]
we have
\[
\sup_{t\in J_i^n} \omega_{f_n(t)} (\delta)\le\omega_{f_n\big(\frac{2\,i+1}{2\,n}\big)}\,\big(\(1+\tfrac1n\)\,\delta\big)\,.
\]
Combining these estimates, one gets by induction that
\[
\sup_{t\in [0,1]} \omega_{f_n(t)}(\delta)\le\omega_{f_0}\(\(1+\tfrac1n\)^n\,\delta\)\le\omega_{f_0}(e\,\delta)\,.
\]
Since $f_0\in\mathrm L^1(\R^d\times\R^d)$, we have $\lim_{\delta \to 0^+}\omega_{f_0}(e\,\delta)=0$. Together with~\eqref{eq:basic0f} and thanks to the Riesz-Fr\'echet-Kolmogorov theorem, we see that $t \in [0,1] \mapsto f_n(t)$ takes its values in a fixed (independent of both $t$ and $n$) compact subset $K$ of $\mathrm L^1(\R^d\times\R^d)$.
\\[4pt]
$\bullet$ \emph{An Ascoli-Arzel\`a argument}. We infer from~\eqref{eq:basic0f} and~\eqref{eq:weakformforfn} that for some explicit constant~$C$, for any $n\in\N$ and $\phi\in X$, one has
\[
\irdxv{\phi\,(f_n(t)-f_n(s))}\le C\,\|\phi\|_X\,|t-s|\quad\forall\,(t,s)\in[0,1]^2\,.
\]
Hence $f_n$ is Lipschitz from $[0,1]$ to $X^*$ uniformly in $n\in\N$. Since $\mathrm L^1(\R^d\times\R^d)$ embeds continuously into $X^*$ and $K$ is compact in $\mathrm L^1(\R^d\times\R^d)$, it is compact in $X^*$ as well so that $(t\mapsto f_n(t))_{n\in\N}$ takes its values in a fixed compact subset of $X^*$. By the Arzel\`a-Ascoli Theorem, we can extract a (not relabeled) subsequence that converges in $\mathcal C([0,T], X^*)$ to some $f$. Because $f$ is Lipschitz from $[0,1]$ to $X^*$ and takes its values in the strongly compact set $K$ in $\mathrm L^1(\R^d\times\R^d)$, this implies that $f\in \mathcal C([0,1],\mathrm L^1\big(\R^d\times\R^d)\big)$. By the same argument $f_n(t)$ converges strongly in $\mathrm L^1(\R^d\times\R^d)$ to $f(t)$ for every~$t$ so that, by Lebesgue's dominated convergence theorem, we obtain $\lim_{n\to+\infty}\|f_n-f\|_{\mathrm L^1((0,1)\times\R^d\times\R^d)}=0$. By the strong $\mathrm L^1(\R^d\times\R^d)$ convergence of $f_n(t)$ to $f(t)$ for every $t$ and~\eqref{eq:basic0f}, we have that $f$ is bounded in $\mathrm L^\infty((0,1);\mathrm L^\infty \cap\mathrm L^1\big(\R^d\times\R^d,(1+|x|^2+|v|^2)^\beta\,dx\,dv\big)$.
\\[4pt]
$\bullet$ \emph{The limit is a weak solution}. It remains to show that the limit $f\in \mathcal C([0,1],\mathrm L^1\big(\R^d\times\R^d)\big)\cap \in\mathrm L^\infty((0,1);\mathrm L^\infty \cap\mathrm L^1\big(\R^d\times\R^d,(1+|x|^2+|v|^2)^\beta\,dx\,dv\big)$ satisfies~\eqref{eq:defweaksol} for every $t\in [0,1]$ and $\phi\in \mathcal{C}_c^2(\R^d\times\R^d)$ by letting $n\to \infty$ in~\eqref{eq:weakformforfn}. Obviously the left-hand side of~\eqref{eq:weakformforfn} converges to $\irdxv{\phi\,f(t)}$ and $s\mapsto \irdxv{v \cdot \nabla_x\phi\,f_n(s)}$, which is a bounded function, converges pointwise and in $\mathrm L^1(0,1)$ to $s\mapsto \irdxv{v \cdot \nabla_x\phi\,f(s)}$. As a consequence, since $2\,\mathbb 1_{J^n}$ converges to $1$ weakly $\star$ in $\mathrm L^\infty(0,1)$, we have
\[
\lim_{n\to \infty} 2 \int_0^t \mathbb 1_{J^n}(s) \irdxv{v\cdot\nabla_x\phi\,f_n(s)}\,ds=\int_0^t \irdxv{v\cdot\nabla_x\phi\,f(s)}\,ds\,.
\]
If $m\in(1, m_2)$, since $f_n(s)$ is bounded in $\mathrm L^\infty$ uniformly in $n$ and $s$, $f_n^m(s)$ converges to $f^m(s)$ strongly in $\mathrm L^1(\R^d\times\R^d)$ for every $s$ and arguing as above, we get
\[\lim_{n\to \infty} 2 \int_0^t \mathbb 1_{I^n}(s) \irdxv{\Delta_v\phi\,f_n^m(s)}\,ds= \int_0^t \mathbb 1_{I^n}(s) \irdxv{\Delta_v\phi\,f^m(s)}\,ds\,.
\]
If $m\in(m_1,1)$, $f_n^m(s)$ converges to $f^m(s)$ strongly in $\mathrm L^{1/m}(\R^d\times\R^d)$. Hence $\Delta_v\phi\,f_n^m(s)$ converges to $\Delta_v\phi\,f^m(s)$ strongly in $\mathrm L^1(\R^d\times\R^d)$ for every $s$ and we reach the same conclusion as in the previous case. Passing to the limit in~\eqref{eq:weakformforfn}, we conclude that $f$ solves~\eqref{Eqn:f}.
\end{proof}

\subsection{\texorpdfstring{Comparison, $\mathrm L^1$}{L1}-contraction and consequences}

Existence of solutions for~\eqref{Eqn:g} is equivalent to existence of solutions for~\eqref{Eqn:f} by the change of variables~\eqref{SSCoV}-\eqref{R}. In the rescaled setting, Barenblatt-Pattle profiles provide us with stationary barrier functions and uniqueness can be deduced from an $\mathrm L^1$ contraction principle,~see Lemma \ref{Lem:L1contraction}, below.
\begin{proposition}\label{Prop:Uniqueness} Under the assumptions of Theorem~\ref{Thm:Main},~\eqref{Eqn:g}-\eqref{g0} has at most one solution in $\mathcal C\big(\R^+;\mathrm L^1(\R^d\times\R^d)\big)$ with
\[
g_{\gamma_1}(x,v)\le g(t,x,v)\le g_{\gamma_2}(x,v)\quad\mbox{and}\quad\|g(t,\cdot,\cdot)\|_{\mathrm L^1(\R^d\times\R^d)}=1\,.
\]
\end{proposition}
Any solution $f\in\mathcal C\big(\R^+;\mathrm L^1(\R^d\times\R^d)\big)$ of~\eqref{Eqn:f} can be rewritten by solving the characteristics of the free transport operator, \emph{i.e.}, by solving
\be{Eqn:ht}
\partial_th=\(\Delta_v-2\,t\,\nabla_x\cdot\nabla_v+t^2\,\Delta_x\)h^m
\ee
where $h(t,x,v)=f(t,x+t\,v,v)$. Even if~\eqref{Eqn:ht} is degenerate, classical methods for porous media ($m>1$) and fast diffusion ($m<1$) equations apply. We refer to~\cite{MR2282669,MR2286292,MR797051,Blanchet:2009sf} for further results on nonlinear diffusion equations in $\mathrm L^1_{\rm{loc}}(\R^d\times\R^d)$. Because fundamental solutions are positive with finite mass, we shall only consider solutions of~\eqref{Eqn:f} with finite $\mathrm L^1(\R^d\times\R^d)$ norm and nonnegative initial data.
\begin{lemma}\label{Lem:L1contraction} Let $d\ge1$ and $m\in(m_1,1)\cup(1,m_2)$. If $f_1$ and $f_2$ solve~\eqref{Eqn:f} in $\mathcal C\big(\R^+;\mathrm L^1(\R^d\times\R^d)\big)$ as in Theorem~\ref{Thm:Main}, then $t\mapsto\irdxv{\big(f_2(t,x,v)-f_1(t,x,v)\big)_+}$ is nonincreasing.\end{lemma}
\noindent Here $h_+:=\big(h+|h|\big)/2$ denotes the positive part of $h$. The result follows by standard methods of approximation as in~\cite{MR797051} or~\cite[Section~2.1]{Blanchet:2009sf}. A straightforward consequence of Lemma~\ref{Lem:L1contraction} is the following \emph{Maximum Principle}.
\begin{corollary}\label{Cor:MaximumPrinciple} Let $d\ge1$ and $m\in(m_1,1)\cup(1,m_2)$. If $f_1$ and $f_2$ solve~\eqref{Eqn:f} in $\mathcal C\big(\R^+;\mathrm L^1(\R^d\times\R^d)\big)$, respectively with initial data $f_{1,0}$ and $f_{2,0}$ which satisfy~\eqref{Initial:f} and such that $f_{1,0}\le f_{2,0}$ a.e., then $f_1(t,\cdot,\cdot)\le f_2(t,\cdot,\cdot)$ a.e.~for all $t\in\R^+$.\end{corollary}
\noindent With Corollary~\ref{Cor:MaximumPrinciple}, the proof of Proposition~\ref{Prop:Uniqueness} is now classical: see for instance~\cite{Blanchet:2009sf}.

\subsection{Scalings and the four-step method}

A possible strategy for proving Theorem~\ref{Thm:Main} is to adapt the \emph{four-step method} introduced by S.~Kamin and J.L.~V\'azquez in~\cite{MR1028745} and summarized for instance in~\cite[Section~18.2]{MR2286292}, as follows.
\begin{enumerate}
\item\emph{Rescaling} of the equation by any parameter $\lambda>0$ producing new solutions $f_\lambda$.
\item\emph{Uniform estimates and compactness} of the rescaled family $(f_\lambda)_{\lambda>0}$.
\item\emph{Passage to the limit} as $\lambda\to+\infty$ (up to the extraction of subsequences).
\item\emph{Identification of the limit} as a Barenblatt-Pattle solution.
\end{enumerate}
Scaling properties of the equation have already been taken into account in the self-similar change of variables~\eqref{SSCoV}. With $A$ as in~\eqref{A}, notice in particular that
\be{fKV}
f_\star(t,x,v)=t^{-d\,\frac{1+A}{1-A}}\,f_\star\(1,t^{-\frac1{1-A}}\,x,t^{-\frac A{1-A}}\,v\)
\ee
for all $t>0$. Within our framework, it can be expected that regularity properties derive from~\eqref{Eqn:ht} using standard parabolic methods for nonlinear diffusions and hypoelliptic techniques. However, to our knowledge, such a theory is still to be done. Compactness properties would then follow. The limit as $t\to+\infty$ of the solution has to be one of the self-similar solutions for~\eqref{Eqn:f} with unit mass. After the change of variables~\eqref{R}, we shall prove that the unique possible limit turns out to be~$g_\star$ given by~\eqref{gstar}. In order to prove this fact and establish Theorem~\ref{Thm:Main}, we ultimately rely on entropy methods, which are well adapted to kinetic equations.

Incidentally, we can observe that~\eqref{Eqn:f} admits a mass-preserving scale invariance $f\mapsto f_\lambda$ for any $\lambda>0$, given by
\[
f_\lambda(t,x,v):=\lambda^4\,f\(\lambda^{2\,(m-m_1)}\,t,\,\lambda^{m-m_3}\,x,\,\lambda^{m_2-m}\,v\)\,,
\]
$m_1=1-1/d$ and $m_2=1+1/d$ as in~\eqref{Range} and $m_3:=1-3/d$. We recover~\eqref{fKV} with the ansatz $\lambda^{2\,(m-m_1)}\,t=1$. Combined with the mass invariance of Lemma~\ref{Lem:mass}, this paves the way for relaxing~\eqref{Initial:f} to a more general condition on the behaviour of $f_0\in\mathrm L^1(\R^d\times\R^d)$ as $|x|^2+|v|^2\to+\infty$ than~\eqref{Initial:f}. See~\cite{bouin2026nonlinearkineticfokkerplanckequations,hao2026kineticfokkerplanckequationsnonlinear} for partial results in this direction.

\section{Entropy methods in self-similar variables}\label{Sec:Entropy}

\subsection{Moments and a lower bound}\label{Sec:Moments}

According to Proposition~\ref{Prop:Uniqueness}, $g_\star$ has finite mass for any $m>m_1$. However, $g_\star$ admits a moment in $\(|x|^2+|v|^2\)$ only if
\be{tilde:m1}
m>\widetilde m_1=\frac d{d+1}>m_1\,.
\ee
The missing range is covered by considering the difference with~$g_\star$.
\begin{lemma}\label{Lem:Moments} Under the assumptions of Theorem~\ref{Thm:Main}, with $g_\star$ given by~\eqref{gstar}, if \hbox{$m_1<m\le\widetilde m_1$}, the solution~$g$ of~\eqref{Eqn:g}-\eqref{g0} is such that
\[
\sup_{t\ge0}\irdxv{\(|x|^2+|v|^2\)\big|g(t,x,v)-g_\star(x,v)\big|}<\infty\,.
\]\end{lemma}
\begin{proof} This is a consequence of Proposition~\ref{Prop:Uniqueness} and
\[
|g(t,x,v)-g_\star(x,v)|\le g_{\gamma_2}(x,v)-g_{\gamma_1}(x,v)=O\(\big(|x|^2+|v|^2\big)^{\frac{2-m}{m-1}}\)\quad\mbox{as}\quad|x|^2+|v|^2\to+\infty\,.
\]
Against $\(|x|^2+|v|^2\)$, the right-hand side is indeed integrable if and only if $m>m_1$.\end{proof}

Let us define the \emph{relative entropy} functional by
\[
\mathcal E[g]:=\mathcal H[g]-\mathcal H[g_\star]\quad\mbox{with}\quad\mathcal H[g]:=\irdxv{\(\frac{g^m}{m-1}+\tfrac{1+A}2\,\big(|v|^2+A\,|x|^2\big)g\)}
\]
for any $m$ satisfying~\eqref{Range}, and with the help of Lemma~\ref{Lem:Moments} if $m\le\widetilde m_1$. This definition generalizes the entropy introduced in~\cite{Newman1984,Ralston1984,MR1940370} for nonlinear diffusions and was used for instance in~\cite{MR2338354} for a kinetic relaxation model with same local equilibria as~\eqref{Eqn:g}. We shall write $\mathcal E(t)=\mathcal E[g(t,\cdot,\cdot)]$ if $g$ solves~\eqref{Eqn:g}. If $m<1$, notice that
\[
\mathcal E[g]=\frac1{m-1}\irdxv{\(g^m-g_\star^m-m\,g_\star^{m-1}\,\big(g-g_\star\big)\)}\,.
\]
In the range $m_1<m\le\widetilde m_1$, the relative entropy can be understood as the functional acting on $w=g/g_\star$ given by
\[
w\mapsto\frac1{m-1}\irdxv{\!\(w^m-1-m\,\big(w-1\big)\)g_\star^m}\,.
\]
\emph{Cf.}~\cite{Blanchet:2009sf} for details. As a straightforward consequence of the convexity of $w\mapsto w^m/(m-1)$, we also have the following lower bound.
\begin{lemma}\label{Lem:Bound} Under the assumptions of Theorem~\ref{Thm:Main}, if $g$ solves~\eqref{Eqn:g}-\eqref{g0}, then $\mathcal E(t)$ is nonnegative for all $t\ge0$. \end{lemma}

\subsection{Relative entropy estimates}\label{Sec:RelativeEntropy}

Let us consider a nonnegative solution $g$ of~\eqref{Eqn:g} and the pressure variables
\be{Q}
\mathsf Q:=\frac m{1-m}\,g^{m-1}\quad\mbox{and}\quad\mathsf Q_\star:=\gammastar+\tfrac{1+A}2\(\,|v|^2+A\,|x|^2\,\)
\ee
where $\mathsf Q_\star$ is the pressure associated to $g_\star$. As for the pressure variable $\mathsf P$ introduced in Section~\ref{Sec:Intro}, we take the convention that $\mathsf Q$ is nonnegative in the fast diffusion range $m<1$.
\begin{lemma}\label{Lem:EEP} Under the assumptions of Theorem~\ref{Thm:Main}, if $g$ solves~\eqref{Eqn:g}-\eqref{g0}, then
\be{dEdt}
\frac{d\mathcal E}{dt}\le-\irdxv{g\,\big|\nabla_v\mathsf Q-\nabla_v\mathsf Q_\star\big|^2}\quad\forall\,t>0\,.
\ee
\end{lemma}
\begin{proof} By Lemma~\ref{Lem:Moments}, $\big(g(t,\cdot,\cdot)-g_\star\big)_{t>0}$ is bounded in $\mathrm L^1\big(\R^d\times\R^d,\,\big(|x|^2+|v|^2\big)\,dx\,dv\big)$, uniformly in $t$. A direct computation using
\[
\Delta_vg^m+(1+A)\,\nabla_v\cdot(v\,g)=\nabla_v\cdot\(g\,\big(\nabla_v\mathsf Q_\star-\nabla_v\mathsf Q\big)\)
\]
shows that
\begin{multline*}
\frac1{m-1}\,\frac d{dt}\irdxv{\(g^m-m\,g_\star^{m-1}\,g\)}\\
=\irdxv{\big(\mathsf Q_\star-\mathsf Q\big)\,\nabla_v\cdot\(g\,\big(\nabla_v\mathsf Q_\star-\nabla_v\mathsf Q\big)\)}\,.
\end{multline*}
An integration by parts completes the proof as in~\cite{Blanchet:2009sf} if $m<1$.

If $m>1$, the result is proved by regularizing~\eqref{Eqn:g} as follows. We introduce an additional linear diffusion operator $\varepsilon\,\Delta_vg$ and consider the modified relative entropy based on $w\mapsto w^m/(m-1)+\varepsilon\,w\,\log w$. After computing the corresponding relative Fisher information, we pass to the limit as $\varepsilon\to0_+$ and obtain Inequality~\eqref{dEdt}.\end{proof}

\subsection{Convergence in \texorpdfstring{$\mathrm L^1(\R^d\times\R^d)$}{L1} and intermediate asymptotics}\label{Sec:MainProof}

The proof of Theorem~\ref{Thm:Main} has to be completed by identifying the limit as $t\to+\infty$ of the solution of~\eqref{Eqn:g}: we will prove that
\[
\lim_{t\to+\infty}\nrmxv{g(t,\cdot,\cdot)-g_\star}1=0\quad\mbox{and}\quad\lim_{t\to+\infty}\mathcal E(t)=0
\]
if $g$ solves~\eqref{Eqn:g}-\eqref{g0} under Assumption~\eqref{Initial:f}.

\begin{proof}[Proof of Theorem~\ref{Thm:Main}] By Proposition~\ref{Prop:Uniqueness}, $g$ is uniformly bounded on $\R^+\times\R^d\times\R^d$ with bounded moments for any $t\ge0$. Lemma~\ref{Lem:Bound} and Lemma~\ref{Lem:EEP} yield that $t\mapsto\mathcal E(t)$ decreases to a limit and
\be{eq:intfisherg}
\lim_{t\to+\infty}\int_t^{+\infty}\irdxv{g\,\big|\nabla_v\mathsf Q-\nabla_v\mathsf Q_\star\big|^2}\,ds=0\,.
\ee
Let us define the \emph{spatial density} of $g$ by
\be{rho}
\rho_g(t,x):=\irdv{g(t,x,v)}
\ee
and the corresponding \emph{local equilibrium} by
\[
\widetilde g(t,x,v):=\big(\mu(t,x)+\tfrac{1-m}{2\,m}\,(1+A)\,|v|^2\big)_+^{\frac1{m-1}}
\]
where $\mu$ is found by inverting
\[
\rho_g(t,x)=\irdv{\Big(\mu(t,x)+\frac{1-m}{2\,m}\,(1+A)\,|v|^2\Big)_+^{\frac1{m-1}}}\,,
\]
\emph{i.e.}, $\mu(t,x)=\mu_1\,\big(\rho_g(t,x)\big)^{k-1}$ for some $\mu_1>0$ with $k$ such that
\be{k}
\frac1{k-1}=\frac d2+\frac1{m-1}\,.
\ee
Notice that $\rho_g=\rho_{\widetilde g}$ is uniformly bounded on $\R^+\times\R^d$ by Proposition~\ref{Prop:Uniqueness}.

\smallskip\noindent$\bullet$ \emph{Compactness for the spatial density}. For any increasing sequence of positive times, $(t_n)_{n\in\N}$ such that $\lim_{n\to+\infty}t_n=+\infty$, let us set
\[
g_n(t,x,v):=g(t_n+t,x,v)\,,\quad\widetilde g_n(t,x,v):=\widetilde g(t_n+t,x,v)\,,\quad\rho_n:=\rho_{g_n}\,.
\]
We deduce from~\eqref{Eqn:g} that $g_n$ satisfies
\be{eq:g_n}
\partial_t g_n+ v\cdot \nabla_x g_n= \nabla_v \cdot E_n \quad\mbox{with}\quad E_n=\nabla_v g_n^m+ \big(A\,x +(1+A)\,v\big)\,g_n
\ee
on $(0,1)\times\R^d\times\R^d$. Let $R>0$ and $\psi_R\in\mathcal C_c^\infty(\R^d)$ be a smooth cut-off function such that $0\le\psi_R\le1$, $\psi_R=1$ on $B_R$ and $\psi_R=0$ on $\R^d\setminus B_{R+1}$. Setting $\bar{\psi}_R(x,v):=\psi_R(x)\,\psi_R(v)$, the function $g_n^R(t,x,v):=\bar{\psi}_R(x,v)\,g_n(t,x,v)$ is uniformly bounded and compactly supported in $(x,v)$ and solve
\[
\partial_t g_n^R+ v\cdot \nabla_x g_n^R=\nabla_v \cdot\big(\bar{\psi}_R\,E_n\big)-\nabla_v \bar{\psi}_R \cdot E_n + \big(v\cdot \nabla_x \bar{\psi}_R\big)\,g_n\,.
\]
For each $R$, the right-hand side of the previous equation is, uniformly in $n$, bounded in $\mathrm{L^2}\big((0,1)\times\R^d; \mathrm{H}^{-1}(\R^d)\big)$. It follows from classical \emph{kinetic averaging lemmas} (see~\cite[Theorem~5]{MR1003433}) that the sequence
\[
\rho_n^R(t,x):=\psi_R(x)\irdv{g_n(t,x,v)\,\psi_R(v)}
\]
is (again for each $R$) bounded in $\mathrm H^{1/4}\big((0,1)\times\R^d\big)$. Taking an increasing sequence of radii $R_k$ tending to $+\infty$, by a diagonal extraction, and possibly passing to subsequences, we may assume that for each $k$, $\rho_n^{R_k}$ converges as $n\to \infty$, a.e.~and in $\mathrm L^1\big((0,1)\times\R^d\big)$ to some~$\rho^k$. Noting that $k\mapsto \rho_n^{R_k}$ is nondecreasing, by monotone convergence, we have that $\rho^k$ converges a.e.~and in $\mathrm L^1\big((0,1)\times\R^d\big)$ to $\rho_\infty=\sup_k \rho^k$. Next we observe that \hbox{$\rho_{g_n}=\rho_n^R+ \varepsilon_n^R$} where, thanks to Proposition~\ref{Prop:Uniqueness}, we have
\[
\| \varepsilon_n^R\|_{\mathrm L^1((0,1)\times\R^d)}\leq \iint_{(\R^d\times (\R^d\setminus B_R))\cup ((\R^d\setminus B_R)\times\R^d) } g_{\gamma_2}\,dx\,dv\]
so that
\[
\lim_{R\to \infty} \sup_n \|\varepsilon_n^R\|_{\mathrm L^1((0,1)\times\R^d)} =0\,.
\]
After extraction of a subsequence, $\rho_n=\rho_{g_n}$ converges to $\rho_\infty$ a.e.~and in \hbox{$\mathrm L^1\big((0,1)\times\R^d\big)$}.

\medskip Let us distinguish the fast diffusion case $m<1$ and the porous medium case $m>1$.

\smallskip\noindent$\bullet$ \emph{Fast diffusion case}.
If $m\in (m_1,1)$, we have:
\[
\nrmv{g(t,x, \cdot)-\widetilde g(t,x, \cdot)}1^2\le\frac8m\irdv{\widetilde g^{2-m}}\,\frac1{m-1}\irdv{\(g^m-\widetilde g^m-m\,\widetilde g^{m-1}\,\big(g-\widetilde g\big)\)}
\]
by the Csisz\'ar-Kullback inequality~\cite[Section~2.4.6]{BDNS2025}. With the estimate
\[
\irdv{\widetilde g(t,x,v)^{2-m}}\le\nrm{\widetilde g(t,x,\cdot)}{\mathrm L^\infty(\R^d)}^{1-m}\irdv{\widetilde g(t,x,v)}=\frac{\rho_g(t,x)}{\mu(t,x)}=\frac1{\mu_1}\,\rho_g(t,x)^{2-k}
\]
and the entropy-entropy production inequality
\[
\frac m{1-m}\irdv{g\,\big|\nabla_vg^{m-1}-\nabla_v\widetilde g^{m-1}\big|^2}\ge\frac4{m-1}\irdv{\(g^m-\widetilde g^m(t,x,v)-m\,\widetilde g^{m-1}\,\big(g-\widetilde g\big)\)}
\]
(see~\cite[Lemma~1.12]{BDNS2025}) which amounts to a Gagliardo-Niren\-berg inequality in the $v$ variable only, according to~\cite{MR1940370}, we obtain
\[
\nrmv{g(t,x,\cdot)-\widetilde g(t,x,\cdot)}1^2\le\mathcal C\,\big(\rho_g(t,x)\big)^{2-k}\irdv{g\,\big|\nabla_v\mathsf Q-\nabla_vQ_\star\big|^2}
\]
for a.e.~$(t,x)$ and for a constant $\mathcal C>0$ depending only on $d$ and $m$, thanks to $\nabla_vQ_\star=\tfrac m{1-m}\,\nabla_v\widetilde g^{m-1}$. Since $\rho_g$ is bounded in $\mathrm L^\infty\big(\R^+;\mathrm L^1\cap\mathrm L^\infty(\R^d)\big)$, we obtain
\be{eq:L1convf}
\lim_{t\to+\infty}\int_t^{t+1}\nrmxv{g(s,\cdot,\cdot)-\widetilde g(s,\cdot,\cdot)}1^2\,ds=0\,.
\ee

Writing $Q_n:=\frac m{1-m}\,g_n^{m-1}$, we deduce from~\eqref{eq:intfisherg} that
\be{eqn:fishergnlim}
\lim_{n\to+\infty}\big\|\sqrt{g_n}\,\big(\nabla_vQ_n-\nabla_vQ_\star\big)\big\|_{\mathrm L^2((0,1)\times\R^d\times\R^d)}=0
\ee
and from~\eqref{eq:L1convf}, that
\be{eqn:gntildegnlim}
\lim_{n\to+\infty}\nrm{g_n-\widetilde g_n}{\mathrm L^1((0,1)\times\R^d\times\R^d)}=0\,.
\ee
We have seen that $\rho_n$ converges a.e.~and in $\mathrm L^1\big((0,1)\times\R^d\big)$ to $\rho_\infty$, moreover, we deduce from Proposition~\ref{Prop:Uniqueness} that $\rho_n \geq \rho_{g_{\gamma_1}}>0$, so that $\mu_n=\mu_1\,\rho_n^{k-1}$ converges a.e.~to \hbox{$\mu_{\infty}=\mu_1\,\rho_\infty^{k-1}$} and $\widetilde g_n$ (as well as $g_n$ by~\eqref{eqn:gntildegnlim}) converges a.e.~to
\be{eq:ginfiniloc}
g_\infty(t,x,v)=\(\mu_\infty(t,x)+\tfrac{1-m}{2\,m}\,(1+A)\,|v|^2\)^{\frac1{m-1}}.
\ee
Observing that $| \tilde g_n-g_\infty| \leq g_{\gamma_2}-g_{\gamma_1}$, we deduce from Lebesgue's dominated convergence Theorem that $\widetilde g_n$ (and $g_n$) also converge in $\mathrm L^1$ to $g_\infty$. Using~\eqref{eqn:fishergnlim} and the fact that $g_n$ is bounded in $\mathrm L^\infty\big((0,1)\times\R^d\times\R^d\big)$, passing to the limit in~\eqref{eq:g_n} and recalling that $\nabla_v \cdot E_n$ can be written as $A\,x\,\nabla_v g_n+ \nabla_v \cdot \big(g_n\,(\nabla_v Q_\star-\nabla_v Q_n)\big)$, we infer that
\[
\partial_tg_\infty+v\cdot\nabla_xg_\infty-A\,x\cdot\nabla_vg_\infty=0\,.
\]
We learn from~\eqref{eq:ginfiniloc} that~$g_\infty$ is radially symmetric in~$v$ for every fixed $(t,x)\in(0,1)\times\R^d$. Then, the term $v\cdot\nabla_x g_\infty$ has zero velocity average. By integrating the equation in $v$, and using Stoke's formula on the term $x\cdot \nabla_v g_\infty$, we readily get $\partial_t \rho_\infty=0$. So, $\mu_\infty$ only depends on $x$ and $\partial_t g_\infty=0=v\cdot\nabla_xg_\infty-A\,x\cdot\nabla_vg_\infty$, yielding $m\,\nabla_x\mu_\infty(x)=(1-m)\,(1+A)\,A\,x$. Since $g_\infty$ and $g_\star$ have same mass, we conclude that $g_\infty=g_\star$. This limit is unique, proving that
\[
\lim_{n\to+\infty}\|g_n(t,\cdot,\cdot)-g_\star\|_{\mathrm L^1((0,1)\times\R^d\times\R^d)}=0
\]
along any sequence $(t_n)_{n\in\N}$. By dominated convergence and monotonicity of $t\mapsto\mathcal E(t)$, we deduce $\lim_{t\to+\infty}\mathcal E(t)=0$, yielding $g(t,.,.)\to g_\star$, a.e.~and in $\mathrm L^1(\R^d\times\R^d)$, as $t\to+\infty$.

\smallskip\noindent$\bullet$ \emph{Porous medium case}. If $m\in (1, m_2)$, by arguing as in the case $m<1$, we find (possibly after the extraction of a subsequence) that $\rho_n$, and $\widetilde g_n$ converge a.e.~and in $\mathrm L^1\big((0,1)\times\R^d\times\R^d\big)$ to some limits $\rho_\infty$ and $g_\infty$, the latter being of the form
\[
g_\infty(t,x,v)=\big(\mu_\infty(t,x)+\tfrac{1-m}{2\,m}\,(1+A)\,|v|^2\big)_+^{\frac1{m-1}}\,.
\]
If $1<m<\min\big\{m_2,3/2\big\}$, one has to replace $\nrmxv{g-\widetilde g}1$ in~\eqref{eq:L1convf} by
\[
\irdxv{\big|g(s,\cdot,\cdot)-\widetilde g(s,\cdot,\cdot)\big|\,\widetilde g^{m-1}(s,\cdot,\cdot)}
\]
in order to use the generalized Csisz\'ar-Kullback inequality of~\cite[Proposition~15,~(ii)]{MR1940370}. Doing so, we deduce that $\| (g_n-\widetilde g_n)g_\infty^{m-1}\|_{\mathrm L^1((0,1)\times\R^d)}$ converges to $0$ hence (possibly after a further extraction of a subsequence) $g_n$ converges to $g_\infty$ a.e.~on $\{g_\infty>0\}$ but a simple mass balance argument shows that the same holds on $\{g_\infty=0\}$ as well. We deduce as before the convergence of $g_n$ to $g_\infty$ both a.e.~and in $\mathrm L^1$. By the very same arguments as in the case $m<1$, we find that $\mu_{\infty}$ only depends on $x$ and $v \cdot \nabla_x g_\infty-Ax \cdot \nabla_v g_\infty=0$ implying that on every connected component $\Gamma$ of $\{\mu_\infty>0\}=\{\rho_\infty >0\}$, $\mu_\infty$ is of the form $C_{\gamma}-\frac{m-1}{m}\,(1+A)\,A\,|x|^2$. Since $\mu_\infty$ is radially decreasing, we deduce that $\Gamma$ is a ball. Finally, the fact that $g_\infty$ and $g_\star$ have the same total mass enable us to conclude as before.
The case $m\ge3/2$ requires a more detailed analysis to control for a.e.~ $(s,x)$ the mass of $g(s,x,\cdot)$ on $\R^d\setminus\mathrm{supp}\big(\widetilde g(s,x,\cdot)\big)$, that can be found in~\cite[Appendix A]{bouin2026nonlinearkineticfokkerplanckequations}.
\end{proof}
\noindent By H\"older interpolation between $\mathrm L^1(\R^d\times\R^d)$ and $\mathrm L^\infty(\R^d\times\R^d)$, and the change of variables~\eqref{SSCoV}, we can obtain more general \emph{intermediate asymptotics} than in Theorem~\ref{Thm:Main}.
\begin{corollary}\label{Cor:IntermediateAsymptotics} Under the assumptions of Theorem~\ref{Thm:Main}, if $f$ solves~\eqref{Eqn:f} with initial datum~$f_0$ and $p\in[1,\infty)$, then
\[
\lim_{t\to+\infty}t^{d\,\frac{p-1}p\,\frac{1+A}{1-A}}\,\nrmxv{f(t,\cdot,\cdot)-f_\star(t,\cdot,\cdot)}p=0\,.
\]\end{corollary}

\subsection{Decay rates of the spatial density}

Recall that the \emph{spatial density} is defined by~\eqref{rho}. For the sake of simplicity, we argue in rescaled variables, as given by~\eqref{SSCoV}. To ensure that moments in~$v$ are finite, we also impose $m>\widetilde m_1$ as in~\eqref{tilde:m1}.

\medskip Let us recall a classical interpolation inequality in kinetic theory, see~\cite[Corollary~2]{MR1115549}.
\begin{lemma}\label{Lem:ClassicalInterpolation} If $g\in\mathrm L_+^1\big(\R^d\times\R^d\!,|v|^2\,dx\,dv\big)\cap\mathrm L^\infty\big(\R^d\times\R^d\!,dx\,dv)$, then we have
\[
\nrmx{\rho_g}{1+\frac2d}\le\mathcal C_d\,\nrmxv g\infty^\frac2{d+2}\(\irdxv{\kern-14pt|v|^2\,g(x,v)}\)^\frac d{d+2}
\]
with $\mathcal C_d:=2^{d/(d+2)}\,\frac{d+2}{2\,d}\,\big|\mathbb S^{d-1}\big|^\frac2{d+2}$.\end{lemma}
\begin{proof} The proof is elementary. If $|v|\le R$, we estimate $g(t,x,\cdot)$ with $\nrmxv g\infty$, while for $|v|>R$ we use $g\,|v|^2/R^2$, so that
\[
\rho_g(x)\le\nrmxv g\infty\,\big|\mathbb S^{d-1}\big|\,\tfrac{R^d}d+\tfrac1{R^2}\irdv{|v|^2\,g(x,v)}\,.
\]
We minimize the right-hand side on $R>0$ and conclude by integrating with respect to $x\in\R^d$ after taking the power $(d+2)/d$ of both sides of the estimate.\end{proof}

We recall that $\gamma=\gammastar$ in~\eqref{g} so that \hbox{$\irdxv{g_\star}=1$}. Again we distinguish the fast diffusion case and the porous medium case.

\smallskip\noindent$\bullet$ If $m<1$, we can write $\mathcal E[g]=Z_m\irdxv{\phi(g/g_\star)\,g_\star^m}$ with
\[
\phi(s):=\tfrac{Z_m^{-1}}{m-1}\(s^m-1-m\,(s-1)\)\,,\quad Z_m=\irdxv{g_\star^m}\,.
\]
The function $\phi:\R^+\to\R^+$ is invertible if it is restricted to $[1,+\infty)$. Let us denote by $\psi$ this inverse and notice that $\phi(s)\le t$ implies $s\le\psi(t)$ for any $s\in(0,+\infty)$. The following estimate is straightforward, by applying Jensen's inequality to $\phi$.
\begin{lemma}\label{Lem:Jensen} If $m\in(\widetilde m_1,1)$, then we have
\be{v^2:m<1}
(1+A)\,\frac{1-m}{2\,m}\irdxv{|v|^2\,g}\le Z_m\,\psi\(Z_m^{-1}\,\mathcal E[g]\)
\ee
for any $g\in\mathrm L_+^1\big(\R^d\times\R^d,(1+|v|^2)\,dx\,dv\big)\cap\mathrm L^m\big(\R^d\times\R^d\!,dx\,dv)$.
\end{lemma}
\noindent$\bullet$ In the case $m>1$, we shall notice that
\be{v^2:m>1}
\tfrac12\,(1+A)\irdxv{|v|^2\,g}\le\mathcal E[g]+\mathcal H[g_\star]\,.
\ee
If $g$ solves~\eqref{Eqn:g}, we learn from Lemma~\ref{Lem:EEP} and Lemma~\ref{Lem:ClassicalInterpolation} combined with~\eqref{v^2:m<1} and~\eqref{v^2:m>1} that $\nrmx{\rho_g(t,\cdot)}{1+2/d}$ is uniformly bounded in terms of $g_0$.

\smallskip Both for $m<1$ and $m>1$, after taking~\eqref{SSCoV} and~\eqref{R} into account, we have the following
\begin{corollary} Under the assumptions of Theorem~\ref{Thm:Main}, if $f$ is a solution of~\eqref{Eqn:f} with $f_0\in\mathrm L^1_+\big(\R^d\times\R^d,\,(|x|^2+|v|^2)\,dx\,dv\big)\cap\mathrm L^\infty\big(\R^d\times\R^d)$, then for any $t>0$, we have
\[
\nrmx{\rho_f(t,\cdot)}{1+\frac2d}\le\mathcal C[f_0]\,\big(1+(1-A)\,t\big)^{-\frac{3-d+d\,m}{(d+2)\,(m-m_1)}}
\]
for some explicit constant $\mathcal C[f_0]$ and with $A$ as in~\eqref{A}.\end{corollary}

\section{Further results}\label{Sec:Further}

In this section, we collect various additional observations, mostly at formal level, which explain why the fundamental solution defined by~\eqref{fstar} is of interest as a model case for nonlinear kinetic equations beyond our result on intermediate asymptotics.

\subsection{Kolmogorov's solution in the limit as \texorpdfstring{$m\to1$}{m-to-1}}\label{Sec:Kolmogorov}

The computation of A.~Kolmogorov for $m=1$ can be rephrased in terms of the pressure variable $\mathsf P_{\kern-1pt\star}=-\,\log f_\star$ using
\[
\partial_t\mathsf P_{\kern-1pt\star}=\Delta_v\mathsf P_{\kern-1pt\star}-\,|\nabla_v\mathsf P_{\kern-1pt\star}|^2-v\cdot\nabla_x\mathsf P_{\kern-1pt\star}
\]
if $f_\star$ solves~\eqref{Eqn:f} with $m=1$. This equation is solved by
\[
\mathsf P_{\kern-1pt\star}(t,x,v)=2\,d\,\log t-d\,\log\big(2\,\pi/\sqrt3\big)+\frac1t\,|v|^2-\frac3{t^2}\,v\cdot x+\frac3{t^3}\,|x|^2\,,
\]
which gives $f_\star=e^{-\mathsf P_{\kern-1pt\star}}$ as in~\cite{kolmogoroff_zufallige_1934}. This expression can also be recovered by taking the limit as $m\to1$ in~\eqref{eq:Pstar} using $\lim_{m\to1}A(m)=1/3$,~\eqref{Eq:Mass} and $\lim_{m\to1}(1+(1-m)\,s)^{1/(m-1)}=e^{-s}$, to the price of elementary but tedious computations.

\subsection{Formal diffusion limit}

In order to investigate the diffusion limit, or overdamped regime, let us consider the parabolic scaling applied to a solution $f$ of~\eqref{Eqn:f} such that
\[
f(t,x,v)=\(\frac{\varepsilon^{\,2\,\eta}}{R(\varepsilon\,t)}\)^dh_\varepsilon\(\varepsilon^2\,\tau(\varepsilon\,t),\varepsilon^{\eta+1}\,x,\frac{\varepsilon^{\eta-1}\,v}{R(\varepsilon\,t)}\)\,,
\]
as $\varepsilon\to0$, with $2\,(d\,m-d+1)\,\eta=3$, and assume that
\[
R(s)=\(1+s/\alpha\)^\alpha\,\forall\,s\ge0\quad\mbox{with}\quad1/\alpha=d\,m-d+2=d\,(m-m_c)
\]
so that $(\tau,x,v)\mapsto h_\varepsilon(\tau,x,v)$ solves
\be{Eqn:heps}
\varepsilon\,\frac{d\tau}{ds}(s)\,\partial_\tau h_\varepsilon+R(s)\,v\cdot\nabla_xh_\varepsilon=\frac{\sigma(s)}\varepsilon\(\Delta_vh_\varepsilon^m+\nabla_v\cdot(v\,h_\varepsilon)\)
\ee
where $\sigma(s):=\(1+s/\alpha\)^{-1}$, $\tau=\tau(s)$ has to be chosen appropriately, and $s=\varepsilon\,t$. Notice that the rescaling differs from~\eqref{R}. It is not the scaling of~\cite{bouin2026nonlinearkineticfokkerplanckequations} either. Remarkably, we can start directly from~\eqref{Eqn:f} and there is no need to introduce $\varepsilon$ artificially. A rigourous justification based on a Hilbert expansion would require well-prepared initial data, as in~\cite{bouin2026nonlinearkineticfokkerplanckequations}, which is out of the scope of our considerations here.

At formal level, see~\cite{MR2338354,bouin2026nonlinearkineticfokkerplanckequations}, a standard Hilbert expansion shows that
\[
h_\varepsilon(\tau,x,v)\approx H(\tau,x,v):=\(\mu(\tau,x)+\frac{1-m}{2\,m}\,|v|^2\)_+^\frac1{m-1}\,,
\]
where $\mu(\tau,x)=\mu_1\,\big(\rho(\tau,x)\big)^{k-1}$ is obtained from a given spatial density $\rho(\tau,x)$ by solving
\[
\rho(\tau,x)=\irdv{\(\mu(\tau,x)+\tfrac{1-m}{2\,m}\,|v|^2\)_+^\frac1{m-1}}=\(\frac{\mu(\tau,x)}{\mu_1}\)^{\frac d2+\frac1{m-1}}
\]
with $k=1+2\,\alpha\,(m-1)$ as in~\eqref{k} and an explicit numerical constant $\mu_1>0$. Of course, this simply means that $\rho(\tau,x,v)=\irdv{H(\tau,x,v)}$ has to be determined and is expected to be given by the limit of $\rho_\varepsilon(\tau,x):=\irdv{h_\varepsilon(\tau,x,v)}$ as $\varepsilon\to0$. The spatial density $\rho_\varepsilon$ solves the \emph{local mass conservation} equation
\[
\frac1{R}\,\frac{d\tau}{ds}\,\partial_\tau\rho_\varepsilon+\nabla_x\cdot j_\varepsilon=0
\]
obtained by formally integrating~\eqref{Eqn:heps} with respect to $v$, while the flux
\[
j_\varepsilon(\tau,x):=\frac1\varepsilon\irdv{v\,h_\varepsilon(\tau,x,v)}
\]
obeys to
\[
\frac\varepsilon{R}\,\frac{d\tau}{ds}\,\partial_\tau j_\varepsilon+\nabla_x\cdot\irdv{v\otimes v\,h_\varepsilon(\tau,x,v)}=-\,\frac\sigma{R}\,j_\varepsilon\,.
\]
This equation is obtained by multiplying~\eqref{Eqn:heps} by $v$ and formally integrating with respect to $v$. Heuristically, we may drop the $O(\varepsilon)$ term and write, at leading order in $\varepsilon$,
\begin{multline*}
-\,\frac{\sigma(s)}{R(s)}\,j_\varepsilon(\tau,x)\approx\nabla_x\cdot\irdv{v\otimes v\,H(\tau,x,v)}\\
=\frac{1}{d}\,\nabla_x\irdv{|v|^2\(\mu(\tau,x)+\tfrac{1-m}{2\,m}\,|v|^2\)_+^\frac1{m-1}}=\nu_1\,\nabla_x\(\rho(\tau,x)^k\)
\end{multline*}for some explicit numerical constant $\nu_1>0$. Collecting our estimates, we obtain that $\rho_\varepsilon\approx\rho$ solves the nonlinear diffusion equation
\be{Eqn:rho}
\partial_\tau\rho=\Delta_x\rho^k\quad\forall\,(\tau,x)\in\R^+\times\R^d
\ee
in the limit as $\varepsilon\to0$ if we choose the time scale $s\mapsto\tau(s)$ such that
\[
\frac{d\tau}{ds}=\nu_1\,\frac{R(s)^2}{\sigma(s)}
\]
with the condition $\tau(0)=0$, \emph{i.e.},
\[
\tau(t)=\tfrac{\alpha\,\nu_1}{2\,(1+\alpha)}\(\(1+t/\alpha\)^{2\,(1+\alpha)}-1\)\quad\forall\,t\ge0\,.
\]
With this choice, the initial datum for~\eqref{Eqn:rho} is given in terms of~$f_0$
\[
\rho(\tau=0,x)=\varepsilon^{-d\,(\eta+1)}\irdv{f_0\big(\varepsilon^{-(\eta+1)}\,x,v\big)}
\]
where~$f_0$ stands for the initial datum for~\eqref{Eqn:f}. The self-similar Baren\-blatt-Pattle solution~$\rho$ of~\eqref{Eqn:rho} with initial datum $\rho(0,\cdot)=\boldsymbol\delta$ takes the form
\[
\rho(\tau,x)=(\tau/\beta)^{-d\,\beta}\,\rho_\star\big((\tau/\beta)^{-\beta}\,x)\big)
\]
where $1/\beta=d\,(k-1)+2$ and
\[
\rho_\star(x)=\(\mathsf c_\star+\tfrac{1-k}{2\,k}\,|x|^2\)_+^\frac1{k-1}\quad\forall\,x\in\R^d
\]
for some integration constant $\mathsf c_\star>0$ such that \hbox{$\nrmx{\rho_\star}1=1$}. Assuming that $\alpha$ is positive, that is, $m>m_c$, so that $\beta>0$, we notice that $k-1$ and $m-1$ have the same sign because
\[
k-1=2\,\alpha\,(m-1)\,.
\]
However, this formal computation requires at least a second moment, which means $m>\widetilde m_1$, or equivalently, $k>d/(d+2)$. This last condition ensures that $\irdx{|x|^2\,\rho_\star(x)}$ is finite.

\subsection{Translation and Galilean invariance}

Although~\eqref{Eqn:f} is nonlinear, it makes sense to consider solutions with measure valued initial data corresponding to non-centred Dirac distributions. Since~\eqref{Eqn:f} is invariant under translations in position, \emph{i.e.}, $(x,v)\mapsto(x-x_0,v)$, and Galilean transformations, \emph{i.e.}, $(x,v)\mapsto(x-t\,v_0,v-v_0)$, the solutions are simple.
\begin{proposition}\label{Prop:X0v0} The function $t\mapsto f_\star\big(t,x-x_0-t\,v_0,v-v_0\big)$ solves~\eqref{Eqn:f} with measure initial datum $\boldsymbol\delta\big(x-x_0,v-v_0\big)$ for any $(x_0,v_0)\in\R^d\times\R^d$.\end{proposition}

\subsection{Linearization, eigenfunctions and asymptotic rates}\label{Sec:Linear}

If $g_\star$ is given by~\eqref{gstar} and $g$ solves~\eqref{Eqn:g}, let us consider a perturbation $h$ such that $g(t,x,v)=g_\star(x,v)+\varepsilon\,h(t,x,v)$ and consider the limit as \hbox{$\varepsilon\to0$}. At order $\varepsilon$, we obtain the linear evolution equation
\be{Eqn:h}
\partial_th=\mathcal L\,h
\ee
where the operator $\mathcal L$ is defined by
\be{Lh}
\mathcal L\,h:=m\,\Delta_v\big(g_\star^{m-1}\,h\big)+(1+A)\,\nabla_v\cdot(v\,h)-v\cdot\nabla_xh+A\,x\cdot\nabla_vh\,.
\ee
We assume that $\nrmxv{g_\star}1=1$ with $\gammastar$ determined by~\eqref{Eq:Mass} and recall that by~\eqref{g},
\[
m\,\big(g_\star(x,v)\big)^{m-1}:=\Big((1-m)\(\gammastar+\tfrac{1+A}2\,|v|^2+\tfrac{1+A}2\,A\,|x|^2\)\Big)_+\,.
\]
By taking the limit as $\varepsilon\to0$ in Lemma~\ref{Lem:EEP} and collecting terms at order $\varepsilon^2$, we obtain
\be{L:accretive}
\frac d{dt}\irdxv{|h|^2\,g_\star^{m-2}}=\,-\,2\,m\irdxv{g_\star\,\big|\nabla_v\big(g_\star^{m-2}\,h\big)\big|^2}\,,
\ee
which can also be recovered directly from~\eqref{Eqn:h}. The case $m>1$ raises various difficulties due to the compactness of the support of~$g_\star$ and is omitted here.
\begin{lemma}\label{Lem:Spectrum} Let $d\ge1$ and $m\in(m_1,1)$. The operator~$\mathcal L$ defined by~\eqref{Lh} on $\mathrm L^2\big(\R^d\times\R^d,g_\star^{m-2}\,dx\,dv\big)$ is such that
\[
\mathrm{Spec}(\mathcal L)\subset\big\{\lambda\in\C\,:\,\mathrm{Re}(\lambda)\le0\big\}
\]
and the kernel of $\mathcal L$ is generated by $h_0:=g_\star^{2-m}$.\end{lemma}
\begin{proof} A derivative with respect to $\gamma$ of both sides of~\eqref{Eqn:g} written for the stationary solution $g_\gamma$ given by~\eqref{g} provides us with an eigenfunction with eigenvalue $0$, that can be written as $\partial_\gamma g_\gamma$ at $\gamma=\gammastar$. Up to a multiplicative constant, this function is $h_0$. Any function $h$ such that $\mathcal L\,h=0$ satisfies $\nabla_v\(g_\star^{m-2}\,h\)=0$ a.e.~on the support of $g_\star$ by~\eqref{L:accretive}. Hence, we have to solve $\mathcal L\,h=0$ with
\[
h(x,v)=\rho(x)\,h_0(x,v)\quad\forall\,(x,v)\in\R^d\times\R^d
\]
for some function $\rho$. We deduce from $0=\mathcal L\,h=-\,h_0\,v\cdot\nabla_x\rho$ that~$\rho$ is constant on the support of $g_\star$. Testing the equation $\mathcal L\,h=0$ with $\big(|v|^2+A\,|x|^2\big)\,h$ shows that $h\equiv0$ on $\R^d\times\R^d\setminus\mathrm{supp}(g_\star)$. This proves that the kernel of $\mathcal L$ is generated by $h_0$ and $\lambda_0=0$ is an eigenvalue of $\mathcal L$. Finally, from~\eqref{L:accretive}, we read that $\mathcal L$ has no eigenvalue $\lambda\in\C$ such that \hbox{$\mathrm{Re}(\lambda)>0$}, which completes the proof of Lemma~\ref{Lem:Spectrum}.\end{proof}

{}From the invariances of~\eqref{Eqn:f}, we can identify several real eigenvalues $\lambda$ of~$(-\mathcal L)$ such that $\lambda\ge0$. Indeed, by constructing special solutions $g$ of~\eqref{Eqn:g} such that
\[
g(t,x,v)-g_\star(x,v)=e^{-\lambda\,t}\,h(x,v)+o\big(e^{-\lambda\,t}\big)\quad\mbox{as}\quad t\to+\infty\,,
\]
and passing to the limit in the equation at main order, we find explicit solutions of $\mathcal L\,h+\lambda\,h=0$. By considering translations in $x$, Galilean invariance (translations in $v$), and translations in $t$, which all leave~\eqref{Eqn:f} invariant, in the sense that the functions
\[
(t,x,v)\mapsto f(t,x-x_0,v)\,,\quad(t,x,v)\mapsto f(t,x-t\,v_0,v-v_0)\,,\quad(t,x,v)\mapsto f(t+t_0,x,v)\,,
\]
solve~\eqref{Eqn:f} if $f$ is a solution to~\eqref{Eqn:f}, we are able to identify three positive eigenvalues $\lambda_1$, $\lambda_2$ and $\lambda_3$. The indices $i$ of $\lambda_i$ do not reflect any order in their values, which depend on $m$. In practice, we use self-similar variables and consider the solutions to~\eqref{Eqn:f} with initial data $g_\star(x-x_0,v)$, $g_\star(x,v-v_0)$ and $S_0^{-\,d\,(1+A)}\,g_\star\big(x/S_0,v/S_0^{1-A}-x/S_0\big)$ with $S_0=R(t_0)\neq1$, with $(x_0,v_0,t_0)\in\R^d\times\R^d\times\R\setminus\{0\}$. Notice that we can define $R$ by~\eqref{R} with $R(0)=1$ and take $S_0=R(t_0)$ for some $t_0\neq0$, or directly take $t_0=0$ and $R(0)=S_0\neq1$.
\begin{proposition}\label{Prop:Spectrum} Let $d\ge1$ and $m\in(m_1,1)$. The operator~$(-\mathcal L)$ defined by~\eqref{Lh} on $\mathrm L^2\big(\R^d\times\R^d,g_\star^{m-2}\,dx\,dv\big)$ admits the following eigenvalues and corresponding eigenfunctions:
\begin{enumerate}\setcounter{enumi}{-1}
\item $\lambda_0=0$ and $h_0:=g_\star^{2-m}$,
\item $\lambda_1=1-A$ and $h_1:=g_\star^{2-m}$,
\item $\lambda_2=A$ and $h_{2,i}(x,v):=(v_i-x_i)\,h_0(x,v)$ with $i=1\,,2\ldots d$,
\item $\lambda_3=1$ and $h_{3,i}(x,v):=(v_i-A\,x_i)\,h_0(x,v)$ with $i=1\,,2\ldots d$.
\end{enumerate}
\end{proposition}
\begin{proof} Apart $\lambda_0$ and $h_0$ which have been identified in Lemma~\ref{Lem:Spectrum}, we rely on the invariances associated with scalings, translations with respect to $v$, and translations with respect to $x$, in the rescaled variables.

\smallskip\noindent$\bullet$ \emph{A mode based on scalings}. With $R$ as in~\eqref{R} such that $R(t_0)=S_0\neq1$, the function defined by $f(t,x,v)=f_\star(t+t_0,x,v)$ solves~\eqref{Eqn:f} with initial datum
\[
f_0(x,v)=f_\star(t_0,x,v)=S_0^{-d\,(1+A)}\,g_\star\(\frac x{S_0},\frac v{S_0^A}-\frac x{S_0}\)\,.
\]
Using the rescaling of~\eqref{SSCoV}, the corresponding solution $g$ of~\eqref{Eqn:g}-\eqref{g0} is such that
\[
g(t,x,v)=S(t)^{-d\,(1+A)}\,g_\star\(\frac x{S(t)},\frac{v+x}{S(t)^A}-\frac x{S(t)}\)
\]
with $S(t):=\big(1+\(S_0^{1-A}-1\)e^{-(1-A)\,t}\big)^{1/(1-A)}$. Using the fact that \hbox{$\lim_{t\to+\infty}S(t)=1$}, we can identify the $O\big(S(t)-1\big)$ term of $g(t,x,v)$ as an eigenfunction, that is,
\[
d\,(1+A)\,g_\star+x\cdot\nabla_xg_\star+\big(A\,v-(1-A)\,x\big)\cdot\nabla_vg_\star=2\,d\,(1-m)\,h_1
\]
with
\[
\frac{h_1(x,v)}{h_0(x,v)}=\frac Bm\,\Big(\gammastar+(B-A\,C)\,|v|^2+(1-A)\,C\,x\cdot v+A\,(B-C)\,|x|^2\Big)\,,
\]
$B=(1+A)/2$ and $1/C=d\,(1-m)$. Since $S(t)-1\sim e^{-(1-A)\,t}$ as $t\to+\infty$, by writing that
\[
g(t,x,v)=g_\star(x,v)+2\,d\,(1-m)\,h_1(x,v)\,\big(S(t)-1\big)+O\(\big(S(t)-1\big)^2\)\,,
\]
we also identify $\lambda_1=1-A$ as an eigenvalue of $-\,\mathcal L$.

\smallskip\noindent$\bullet$ \emph{A mode based on translations with respect to $v$ in the phase space}. The solution $g$ of~\eqref{Eqn:g} with initial datum $f_0(x,v)=g_\star(x,v-v_0)$ for any $v_0\neq0$ is given up to $o\big(e^{-A\,t}\big)$ terms by
\[
g(t,x,v)\approx g_\star\(x-\tfrac1{1-A}\,e^{-A\,t}\,v_0,v+\tfrac A{1-A}\,e^{-A\,t}\,v_0\)
\]
and we identify $\lambda_2=A$ as an eigenvalue of $-\,\mathcal L$ corresponding to any of the eigenfunctions $A\,\partial_{v_i}g_\star-\partial_{x_i}g_\star$, that is, of $h_{2,i}$, for any coordinate $i=1\,,2\ldots d$.

\smallskip\noindent$\bullet$ \emph{A mode based on translations with respect to $x$ in the phase space}. The solution $g$ of~\eqref{Eqn:g} with initial datum $f_0(x,v)=g_\star(x-x_0,v)$ is given by
$g(t,x,v)=g_\star\big(x-e^{-t}\,x_0,v+e^{-t}\,x_0\big)$ for any $x_0\neq0$ so that we can identify $\lambda_3=1$ as an eigenvalue of~$-\,\mathcal L$ corresponding to any of the eigenfunctions $\partial_{v_i}g_\star-\partial_{x_i}g_\star$ which can also be written as $h_{3,i}$, for any coordinate $i=1\,,2\ldots d$.
\end{proof}
\begin{remark} Let us consider the pressure variable $\mathsf Q$ as in~\eqref{Q} and notice that~\eqref{Eqn:g}, that is,
\[
\partial_tg=\Delta_vg^m-v\cdot\nabla_xg+d\,(1+A)\,g+\big((1+A)\,v+A\,x)\big)\cdot\nabla_vg\,,
\]
can be rewritten as
\be{Eqn:Q}
\partial_t\mathsf Q=(1-m)\,\mathsf Q\,\Delta_v\mathsf Q-\,|\nabla_v\mathsf Q|^2-v\cdot\nabla_x\mathsf Q-d\,(1+A)\,(1-m)\,\mathsf Q+\big((1+A)\,v+A\,x)\big)\cdot\nabla_v\mathsf Q\,.
\ee
In this framework, the computations of Proposition~\ref{Prop:Spectrum} are easy to check. If
\[
\mathsf Q(t,x,v)=\gammastar+\frac12\,(1+A)\,|v-\bar v(t)|^2+\frac12\,A\,(1+A)\,|x-\bar x(t)|^2\,,
\]
then~\eqref{Eqn:Q} amounts to the system of differential equations
\begin{align*}
&\bar x'=\bar v\,,\\
&\bar v'+A\,\bar x+(1+A)\,\bar v=0\,.
\end{align*}
It is then straightforward to check that the set of solutions is spanned by two and only two families of solutions given by
\[
\bar x(t)=-\,A^{-1}\,v_0\,e^{-\,A\,t}\,,\quad\bar v(t)=v_0\,e^{-\,A\,t}\,,
\]
and
\[
\bar x(t)=-\,\bar v(t)=-\,v_0\,e^{-t}\,,\quad\bar v(t)=v_0\,e^{-\,t}\,,
\]
which correspond respectively to $\lambda=\lambda_2=A$ and $\lambda=\lambda_3=1$.\end{remark}
\begin{figure}[ht]
\begin{center}
\includegraphics[width=4cm]{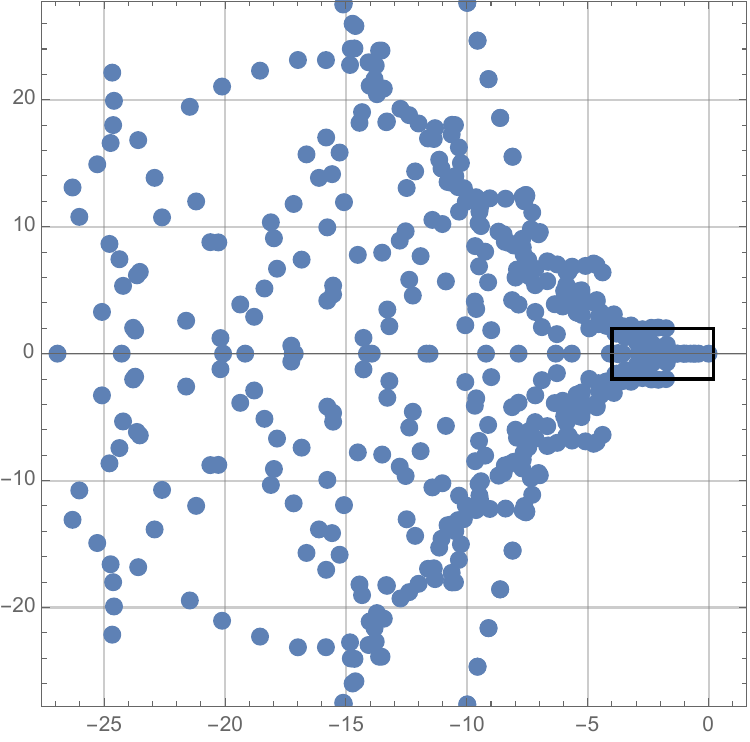}\includegraphics[width=4cm]{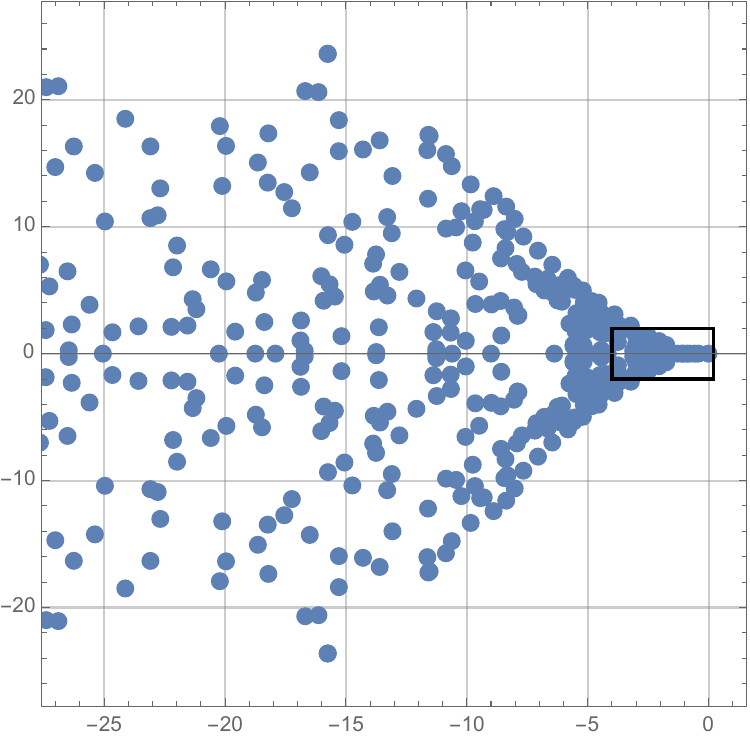}
\includegraphics[width=4cm]{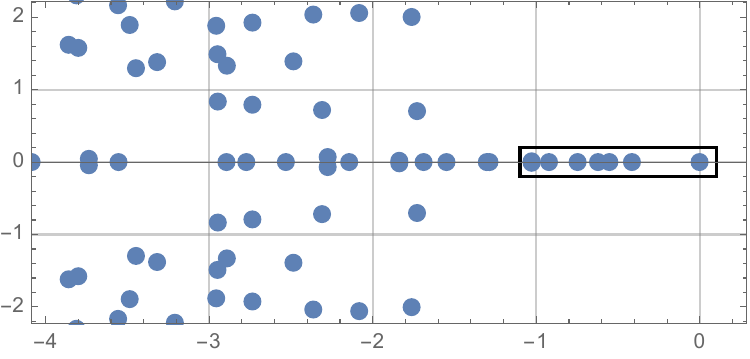}\includegraphics[width=4cm]{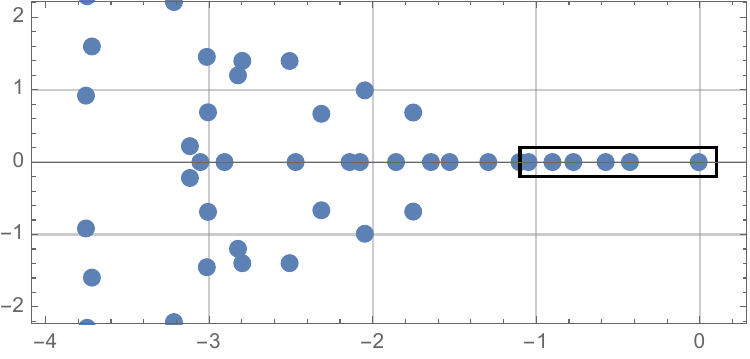}
\includegraphics[width=4cm]{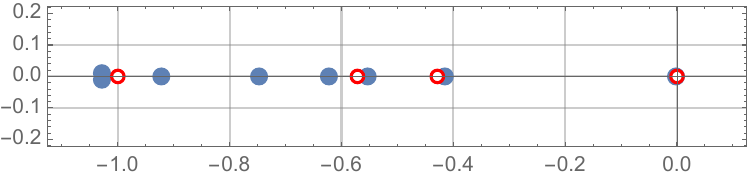}\includegraphics[width=4cm]{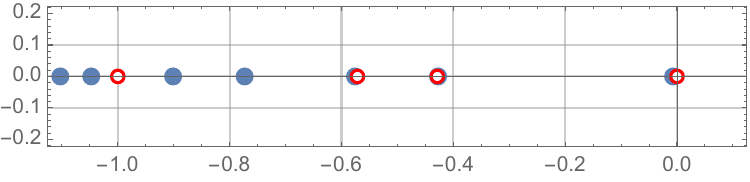}
\caption{\small\label{F2} Numerical computation of the spectrum of the operator $L$ defined by
$\textstyle L\,f=g_\star^{(m-2)/2}\mathcal L\big(g_\star^{(2-m)/2}\,f\big)$
for a smooth function $f$ supported in the rectangle $[-18,18]\times[-28,28]\ni(x,v)$ (left) and in the ellipse (right) with semi-axes $2*18/\sqrt\pi$ for $x$ and $2*28/\sqrt\pi$ for $v$, with zero Dirichlet boundary conditions and $m=0.8$. Here $L$ is defined on $\mathrm L^2(\Omega,dx\,dv)$ where~$\Omega$ is either the rectangle or the ellipse (both have same area). The spectrum of~$L$ is seen as an approximation of the spectrum of $\mathcal L$ on $\mathrm L^2\big(\R^d\times\R^d,g_\star^{m-2}\,dx\,dv\big)$, with dimension $d=1$, for the eigenvalues which are the closest to $0$. Local enlargements close to $\lambda=0$ are shown below the general figure of the spectrum. On the last row, the eigenvalues $-\,\lambda_i$ with $i=0$, $1$, $2$, $3$ are marked by red circles. The highest non-zero eigenvalue of $L$ is of the order of $-\,0.4152$ (left, rectangle) and $-\,0.4272$ (right, ellipse), to be compared with $-\,\lambda_1=-\,A\approx-\,0.4286$.}
\end{center}
\end{figure}
A numerical computation of the spectrum of $\mathcal L$ is shown in Fig.~\ref{F2}, which suggests that, at least for some values of $m$, the operator $\mathcal L$ is sectorial and
\[
\mathrm{Spec}(\mathcal L)\setminus\{0\}\subset\big\{\lambda\in\C\,:\,\mathrm{Re}(\lambda)\le-\,\min\{A,1-A\}\big\}\,.
\]
A simpler question would be to prove that
\[
\sup\big\{\mathrm{Re}(\lambda)\,:\,\lambda\in\mathrm{Spec}(\mathcal L)\setminus\{0\}\big\}<0\,.
\]
Coming back to the nonlinear evolution problem~\eqref{Eqn:g}, it is also an open question to decide whether there is some $\lambda_\star>0$ such that
\be{lambdaStar}
\limsup_{t\to+\infty}e^{\,\lambda_\star t}\,\mathcal E(t)<+\infty
\ee
for any function $g$ satisfying the assumptions of Theorem~\ref{Thm:Main}, or not. One can also wonder whether $\lambda_\star=\min\{A,1-A\}$ or not, and ask if~\eqref{lambdaStar} is also true under more general conditions on the initial datum than in Theorem~\ref{Thm:Main}.

\subsection*{Acknowledgements}{\!This work has been supported by the Project \emph{Conviviality} (ANR-23-CE40-0003) of the French National Research Agency.\\[4pt]
{\small\copyright\,\the\year\ by the authors. This paper may be reproduced, in its entirety, for non-commercial purposes. \href{https://creativecommons.org/licenses/by/4.0/legalcode}{CC-BY 4.0}}}

\bigskip\begin{center}\rule{2cm}{0.5pt}\end{center}\bigskip
\end{document}